\documentclass[final]{siamltex}

\usepackage{graphics}
\usepackage{epsfig}
\usepackage{amssymb,amsmath}
\newtheorem{assumption}[theorem]{Assumption}
\usepackage{subfigure}
\usepackage{color}

\title{Limitations for nonlinear stabilization over uncertain channel}

\author{Umesh Vaidya\thanks{ Dept. of Elec. and Comp. Engg., Iowa State University,
Ames, IA-50011({\tt \small
ugvaidya@iastate.edu}).} \and Nicola Elia
\thanks{Dept. of Elec. and Comp. Engg., Iowa State University,
Ames, IA-50011 ({\tt nelia@iastate.edu}).
This paper is an extended version of the paper ``Limitations for nonlinear stabilization over erasure channel" appeared in the proceedings of IEEE Control and Decision Conference, 2010, Atlanta, GA. }}%\thanks{The two authors have contributed equally to this paper}

\begin{document}

\maketitle

\begin{abstract}
We study the problem of mean-square exponential incremental stabilization of nonlinear systems over  uncertain communication channels. We show  the ability to stabilize a system over such channels is fundamentally limited and  the channel uncertainty must provide a minimal Quality of Service (QoS) to support stabilization. The smallest QoS necessary for stabilization is shown as a function of the positive Lyapunov exponents of uncontrolled nonlinear systems. The positive Lyapunov exponent is a measure of dynamical complexity and captures the rate of exponential divergence of nearby system trajectories.
One of the main highlights of our results is the role played by nonequilibrium dynamics to determine the limitation for incremental stabilization over networks with uncertainty.

%The smallest QoS necessary for stabilization is a function of the Hessian associated to  the deterministic  minimal energy  control problem for the system.
%Under some ergodicity assumptions, we show that the Hessian reduces to the sum of positive Lyapunov exponents and captures the instability of non-equilibrium dynamics of the system. The positive Lyapunov exponents  of the open loop system capture the global instability of the nonlinear systems. Hence the dependence of QoS limitation results highlights, for the first time, the important role played by the global non-equilibrium dynamics of the plant in the stabilization of systems over networks.

%
%Hence the main results of this paper
%highlights the emergence of dynamics away from the equilibrium as a key towards understanding the limitation for nonlinear stabilization over erasure channels.
%

%The dependence of  the result on Lyapunov exponents corresponding to non-trival (other than equilibrium point) dynamics  highlights the emergence of dynamics away from the equilibrium as a key towards understanding the limitation for nonlinear stabilization over erasure channels.
% Existing results for linear time invariant systems are derived as a special case of the nonlinear system results.

\end{abstract}

\section{Introduction}

Networked controlled systems have been the focus of much research in recent years \cite{networksystems_specialissue, communication-channel, sekhar.thesis, anant, hadjicostis, jump_estimator, Luca03kalmanfiltering, michales, martins_elia, nair-ms, savkin1}. Among several relevant questions, one main problem is to characterize the limitations of closed loop stability and performance
induced by the presence of unreliable communication channel(s) in the loop.

The vast majority of the studies have focused on Linear Time-Invariant (LTI) plants under a variety of settings and assumptions.
Tatikonda and Mitter \cite{sekar_TAC} has considered finite-rate channels and showed  the required communication rate for stabilization is function of the unstable eigenvalues of the plant.
Sahai \cite{anant}  considered noisy channel models and showed a new notion of reliable communication, different than Shannon's,
is required to obtain the moment stability of the closed loop.  Besides making clear that different notions of stability may require different notions of reliable
communication in the context of a simple binary erasure channel, \cite{anant} has identified fading in the channel rather than its discrete nature, a main source of
 limitation for reliable communication. In fact, the binary erasure channel can be thought of as a 1-bit finite-rate channel with extreme $0-1$ fading.
 The focus on fading has lead to the introduction of analog erasure channels, which we consider in this paper, where the quantization effects are neglected and the channel is modeled as a random independent
 Bernoulli switch \cite{hadjicostis,jump_estimator,michales}. This model and its generalization to Markovian-dependent switches is a simple, yet effective, model for packet-drops links,
 and allows study of the stability limitations induced by these channels in the limit of large packet length, using classical stochastic systems techniques and tools. It is difficult to summarize the plethora of research in the area of control of LTI systems  involving an independent, identically distributed (IID) Bernoulli switch. The notion of stability, the location of the channel(s)
(at the sensor or actuator side or both) \cite{scl04, networksystems_foundation_sastry, network_basar_tamar}, the information available to encoders and decoders (like channel state information (CSI) or knowledge of other signals in the loop)
and their computational power \cite{scl04,gupta},  and, of course, other performance objectives \cite{sinopoli.acc05,network_basar_tamar,prashanthacc} affect the required minimal Quality of Service (probability of successful delivery of the packets) by the communication
link. For example, in the case of one channel on the sensor side with mean-square stability, no encoding, and CSI used at the decoder, the QoS is a function of all the unstable eigenvalues of the plant \cite{tac11}. If a Kalman filter, which uses the plant input, is allowed as an encoder and the entire estimated state of the system is sent over the channel,
then the minimal QoS is only a function of the plant's  eigenvalues of the largest magnitude \cite{gupta}. Finally, if no encoder is allowed and the stability notion is of bounded second moment instead,
then the limitation is not completely characterized, and the QoS is bounded between the previous two cases \cite{Luca03kalmanfiltering,tac11}.
While the picture is fragmented, at the same time it is clear, channel fading fundamentally limits the ability to stabilize the networked system, at least in term of moments.

For completeness, we  point out  a current body of literature on the related problem of state estimation over noisy channels; see, for example,
\cite{Luca03kalmanfiltering,schwartz_opt_rec, murray_epstein}. When the erasure channel is on the sensor side, like in the estimation problems, the development is somewhat simplified because the channel state
is available at the receiver side and can be included in the design of the estimator \cite{Luca03kalmanfiltering}.
In this case, the design approach for Markovian Jump Linear Systems (MJLS) \cite{control_MJLS_costa} can be quite useful, \cite{ jump_estimator, estimation_MJLS}.
When the channel is on the actuator side, the problem is more involved, as the controller and the decoder, which can use the CSI, are physically separated and their design
is decentralized, in general. Excluding some special cases \cite{tac11}, the controller cannot be an instantaneous function of
the channel state. For example in this paper's setup this makes therefore the MJLS design unamenable to this problem.

While there are many results involving the stabilization of linear systems over communication networks, little has been developed for nonlinear systems with the notable exception by \cite{Nair04, mehta_fund}. The analysis of the limitations for the stabilization of nonlinear systems is reduced through linearization or through robust control theory to the linear case involving the eigenvalues of the linearized system.
The results in \cite{Nair04} and \cite{mehta_fund} have connected the limits for stabilization with topological and measure theoretic (or metric) entropy corresponding to equilibrium dynamics, respectively.
Even less work has  considered fading channels \cite{amit_observation, scl_erasure}. In \cite{scl_erasure}, the limitation of stabilization for a scalar nonlinear system over the fading channel is studied to provide a necessary condition for general $q^{th}$ moment stabilization. The various proofs involved in the vector case are fundamentally different and cannot be trivially extended from the scalar case. Similarly \cite{amit_observation, cdc_erasure,amit_observation_conference}, developed along with this paper, considered the estimation and stabilization problem and used many of the ideas developed in this paper. However, the distinctive feature of the stabilization problem is the limitations for stabilization are shown to be offered by two competing local and global dynamics of the system. Furthermore, unlike results in \cite{amit_observation, cdc_erasure}, results in this paper are proved under the noncompact assumption on the state space and in the presence of additive noise term. In this paper, we also consider more general channel uncertainty model as opposed to Bernoulli (on-off)  channel uncertainty considered in \cite{amit_observation,cdc_erasure}. Establishing connection between various notions of stabilities, in particular the incremental mean square and bounded second moment stability for nonlinear systems evolving on unbounded state space, also form one of the main contributions of this paper.

In this paper, we fill the gap and study the limitations of the stabilization for a class of nonlinear unstable systems over general uncertain channels.
We consider a set-up sufficiently simple  to allow characterization of
the limitations and, at the same time, not too unrealistic. We assume the controller has perfect
access to the plant's  state. However, it can act on the plant's input through an uncertain channel. Moreover, on the receiver's side, there is no sophisticated decoder that exploits the
CSI. This is a simplified model of a User Datagram Protocol (UDP)-like or best effort protocol \cite{network_basar_tamar}.  While more complex schemes can be studied using our approach, their investigation is outside the scope of this paper.
The main contribution of this paper is to show  the open loop, nonequilibrium dynamics of the plant also imposes minimal requirements on the QoS of the link, together with the local instability of the equilibrium point.
This represents an important departure from local equilibrium results and shows the relevance of nonequilibrium dynamics. In particular, it shows a failure to allocate QoS to the unstable nonequilibrium dynamics and only focus on the unstable equilibrium dynamics can have
disastrous consequences.  Thus, the problem of characterizing the limitations of the stabilization and estimation of network nonlinear  systems over noisy, uncertain channels is a timely research
topic, given the important role that nonlinear dynamics play in applications, such as network power systems and biological networks.
Characterizing these fundamental limitations will help understand important tradeoffs between uncertainty and robustness in such complex networked systems.

In this paper, we study the problem of characterizing such fundamental limitations in the stabilization of a class of nonlinear systems controlled
over general uncertain channel. The case of  on/off fading channel with IID Bernoulli fading forms the special case of the more general uncertain channel model considered in this paper. The objective is to characterize the quality of service of the channel in terms of probability of successful transmission to guarantee  mean-square exponential incremental
stability of the closed-loop system.  Mean-square exponential incremental stability implies exponential convergence of all  system trajectories.
This form of stability for nonlinear systems is studied using Lyapunov function and contraction analysis \cite{Incremental_david, slotine_contraction}.
Incremental stability of a nonlinear system is stronger than global stability of the equilibrium point.
In fact, for autonomous nonlinear systems, incremental stability implies global convergence to an equilibrium point.
Incremental stability is shown to play an important role in the problems of synchronization, nonlinear tracking, and nonlinear observer design \cite{Incremental_david, slotine_contraction, Sepulchre_inter}.

The main result of this paper proves  fundamental limitations arise for mean-square exponential incremental stabilization.
The limitations are expressed in terms of the probability of erasure and positive Lyapunov exponents of the uncontrolled open-loop system.
Positive Lyapunov exponents are a measure of dynamical complexity, since they capture the rate of exponential divergence of nearby system trajectories \cite{Ruelle85}.

%Using some of the existing results from ergodic theory of dynamical systems \cite{Mane, ergodic_theory_walter}, we also  show that the minimum QoS is a function of another
%measure of dynamical complexity, i.e., entropy, in particular, the measure-theoretic (or metric) entropy of an open-loop system. This relationship with system entropy allows a
%comparison of our limitation results with the existing Bode-like fundamental limitation results for nonlinear systems, which, too, connects limits for stabilization
%with the entropy of open-loop systems. This comparison leads to one of the main findings of this paper. While the existing Bode-like limitation results
%express a stabilization limit with the entropy of an invariant measure corresponding to the unstable equilibrium point, we prove that the limitations for incremental
%stabilization over erasure channels are presented by the entropy corresponding to an invariant measure capturing the dynamics away from the equilibrium point.
%To the best of the authors knowledge, this is the first systematic results that express limits for stabilization in terms of entropy of nonequilibrium dynamics.
This emergence of nonequilibrium dynamics as the limiting factor for stabilization arises, due to the stronger notion of incremental stability  we use in
this paper and the presence of uncertainty in the form of erasure in the feedback loop. The uncertainty in the feedback loop has
the ability to steer the system state away from the equilibrium point, where the nonequilibrium dynamics of the system dominates \cite{scl_erasure}.

%,  which can be expressed in terms of probability of erasure and the Hessian of the optimal cost function corresponding to the deterministic minimum energy optimal control problem. Under some ergodicity assumptions on the system dynamics, the Hessian of the optimal cost function  reduces to sum of positive Lyapunov exponents, thus capturing the instability of non-equilibrium dynamics of the  plant. The dependence of the QoS limitation on the Lyapunov exponents highlights, for  the first time, the importance of the global non-equilibrium dynamics of the plant in the stabilization of systems over networks.
The paper presents following important innovations: 1) Extends the framework of random dynamical systems \cite{Kifer, arnold_book_rds} to controlled dynamical systems
\cite{scl_erasure,cdc_erasure}. 2) Connects the stability requirement with the Quality of Service of the channel and the positive Lyapunov exponents of the open-loop system. In this sense, the results of the paper generalize those of \cite{scl04} and the Lyapunov exponents emerge as a natural generalization of the
linear system eigenvalues to capture the limitations of nonlinear networked systems. 3) Connects various notion of stochastic stabilities on unbounded state space in particular incremental mean square exponential stability and bounded second moment.
%The Lyapunov exponents have long been thought as the generalization of
%eigenvalues from linear systems to nonlinear systems. However, ours is the first systematic result to show this generalization in the context of nonlinear
%control systems.
%The main results of this paper can also be interpreted in terms of the entropy of the dynamical system. This can be done via Pesin's formula \cite{Mane}, which establishes a relation between the sum of positive Lyapunov exponents and the entropy of a dynamical system. By relating the solution of the minimum energy optimal control problem and the sum of positive Lyapunov exponents, our main results also connect two known notions of entropies from control nonlinear systems \cite{entropy_pabalo} and dynamical systems \cite{Mane}.

The organization of the paper is as follows. Section \ref{preliminary} presents some preliminaries and the stability definition used in this paper. In Section \ref{section_motivation}, we present a motivating example for the results developed in this paper. Section \ref{limitation_section} shows the main result on the performance limitations for mean-square exponential incremental stabilization. The main result from Section \ref{limitation_section} is used to provide limitations for incremental stabilization under various assumptions on system dynamics in Section \ref{section_main}. Simulation results are presented in Section \ref{section_simulation}, followed by conclusions in Section \ref{section_conclusion}.
\section{Preliminaries and definitions}\label{preliminary}
We consider the problem of stabilization of a multi-state multi-input nonlinear system of the form
\begin{equation} x_{n+1}=f(x_n)+ B v_n+\gamma_n, \label{system1}
\end{equation}
where $x_n\in \mathbb{R}^N$ is the state,
$v_n\in \mathbb{R}^d$ is the plant  input, and $\gamma_n\in \mathbb{R}^N$ is assumed to be an IID random variable satisfying following statistics
\[E_{\gamma_n}[\gamma_n]=0,\;\;\;E_{\gamma_n}[\parallel \gamma_n \parallel^2 ]\leq C\]
for some constant $C$. The system mapping, $f$, satisfies the following assumptions.
\begin{assumption}\label{assumption_1}
The system mapping $f: \mathbb{R}^N \rightarrow \mathbb{R}^N$ is assumed to be at least $C^1$ function of $x$ and the
Jacobian $\frac{\partial f}{\partial x}(x)$ is assumed to be invertible and uniformly bounded for almost all (w.r.t. Lebesgue measure) $x\in \mathbb{R}^N$ and for $x=0$.  $x=0$ is an unstable equilibrium point with eigenvalues $|\lambda^i_0|>1$ for $i=1,\ldots N$ of the Jacobian $\frac{\partial f}{\partial x}(0)$. The control matrix, $B$, satisfies $B'B>0$.
\end{assumption}

\begin{remark} The Assumption (\ref{assumption_1}) on all the eigenvalues outside the unit disk is a technical assumption and is made for the simplicity of presentation of the main results of the paper. The proof for the case where some of the eigenvalues are negative can be done by using the technique of tempered transformation \cite{Katok} which allows one to decompose the state space along the directions of stable and unstable manifolds.
\end{remark}
\begin{assumption}\label{assumption_2} Assume  pair $(f(x),B)$ satisfies the following assumption. There exist positive constants $k_1$, $k_2$ and integer $k\geq 0$, such that
\begin{eqnarray}
k_1 I\leq \sum_{\ell=0}^{k} \Phi(x_k,x_{\ell})BB'\Phi'(x_k,x_{\ell}) \leq k_2 I, \label{controllability_cond}
\end{eqnarray}
for almost all with respect to the Lebesgue measure initial condition, $x_0\in \mathbb{R}^N$ and for $x_0=0$, where $x_{n+1}=f(x_n)$, and $\Phi(x_n,x_0):= \frac{\partial f}{\partial x}(x_n)\cdots \frac{\partial f}{\partial x}(x_0)=\prod_{k=0}^n \frac{\partial f}{\partial x}(x_k)$.
\end{assumption}
\begin{remark}
Assumption \ref{assumption_2} corresponds to uniform completely controllability of the linearized system, $\eta_{n+1}=\frac{\partial f}{\partial x}(x_n)\eta_n+B\vartheta_n$, with $\eta_n\in \mathbb{R}^N$, $\vartheta_n\in \mathbb{R}^d$, and $x_{n+1}=f(x_n)$ \cite{kwakernaak}. Condition (\ref{controllability_cond}) can also be related to the Lie-algebra-based controllability condition for discrete-time nonlinear systems \cite{controllability_discrete-time_sontag}.
\end{remark}
In this paper, we consider the following multiplicative channel model,
\begin{eqnarray}
v_n=\xi_n u_n,\label{channel}
\end{eqnarray}
where $\xi_n\in \mathbb{R}$ is an IID random variable with following statistics
\begin{eqnarray}
E[\xi_n]=\mu,\;\;\;E [(\xi_n-\mu)^2]=\sigma^2\neq 0.\label{uncertain_statistics}
\end{eqnarray}
Note that the Bernoulli random variable with $\xi_n\in\{0,1\}$ which is used in the modeling of erasure channels or packet-drop links with a negligible quantization effect will be the special case of more general random variable as defined above with the non-erasure probability $p=\frac{\mu^2}{\mu^2+\sigma^2}$ and  erasure probability $(1-p)=\frac{\sigma^2}{\mu^2+\sigma^2}$. For the special case of Bernoulli random variable, we assume  the controller does not have access to the channel state information.  So, when an erasure exists and $\xi_n=0$, no control input is applied to the plant.
More general settings are definitely of interest and have been considered in the literature in the context of linear systems.
However, the main point is that the limitations of fading links will emerge in all the settings, although with different conditions.
Thus, in this paper we consider the simplest, yet realistic, setup to derive our analysis.

The control input, $u_n$ sent over the uncertain link is assumed  of the form
\begin{eqnarray}
u_n=k(x_n)+w_n,\label{controller1}
\end{eqnarray}
where $k$ is the controller and $w_n$ is the exogenous but deterministic input. All results of this paper are derived under the following assumption on the controller dynamics.
\begin{assumption}\label{assumption_controller} The controller  $k:\mathbb{R}^N\to  \mathbb{R}^d$ is assumed to be memoryless, deterministic, and is at least $C^1$ function of the state $x$ with uniformly bounded Jacobian and $k(0)=0$. Furthermore, $(\frac{\partial f}{\partial x}+\mu B\frac{\partial k}{\partial x})' B(\frac{\partial f}{\partial x}+\mu \frac{\partial k}{\partial x})\geq \alpha I$ for almost all $x\in \mathbb{R}^N$, where $\mu=E[\xi_n]$.
\end{assumption}

Combining (\ref{system1}), (\ref{channel}), and (\ref{controller1})  the feedback control that we study in this paper is represented by following dynamic equation
\begin{eqnarray}
x_{n+1}=f(x_n)+\xi_n Bk(x_n)+\xi_n B w_n+ \gamma_n. \label{system_feedback_noise}
\end{eqnarray}

The objective is to design a state feedback controller, $k(x_n)$, such that the system (\ref{system_feedback_noise}) is mean-square exponential incrementally stable.  In simple terms, this notion of stochastic stability captures the concept that the expected distance (square) between  trajectories of (\ref{system_feedback_noise}) starting from different initial states, but subject to the same sequence of $\xi_n$, $w_n$, and $\gamma_n$ converges to zero exponentially, for any input $w_n$ and almost all initial conditions. To formally state the definition of mean square exponential incremental stability, we define following system with and without additive noise term
\begin{equation}
{\cal S}=\left\{\begin{array}{ccl}x_{n+1}&=&f(x_n)+\xi_n Bk(x_n)+\xi_n B w_n+ \gamma_n\\
x_{n+1}&=&f(x_n)+\xi_n Bk(x_n)+\xi_n B w_n\end{array}\right.\label{big_system}.
\end{equation}
\begin{definition}[Mean-square exponential incrementally stable]\label{definition_stability} 
The system ${\cal S}$ in (\ref{big_system}) is said to be mean
square exponential incremental stable if, 
for almost all Lebesgue measure initial conditions $x_0$ , $y_0\in  \mathbb{R}^N$ and in particular for $x_0 = 0$, for any $w_n$ including $w_n\equiv 0$, and with and without additive noise term $\gamma_n$ , there exist positive constants, $K < \infty$ and $\beta < 1$ such that,
\begin{eqnarray}
E_{{\cal X}_0^n}\left[ \parallel x_{n+1}-y_{n+1}\parallel^2\right]\leq K \beta^n \parallel x_0 -y_0\parallel^2\;\;\;\;\forall n\geq 0.\label{incremental_stable}
\end{eqnarray}
holds true for the trajectories of (\ref{big_system}). The trajectories $\{x_n\}$ and $\{y_n\}$, in \ref{incremental_stable}, start from different initial conditions, but are subjected to the same sequence of $w_n$ and random variables, $\xi_n, \gamma_n$.   The notation $E_{{\cal X}_0^n}[\cdot]$ stands for expectation taken over the sequence(s) of ${\cal X}_0^n$, which is either equal to $\{\xi_0^n , \gamma_0^n \}$ or $\{\xi_0^n\}$ depending upon if there is additive noise term, $\gamma_n$, or not respectively. 
\end{definition}
\begin{remark}  By requiring that (\ref{incremental_stable}) holds true for both the cases of with and without additive noise term and for arbitrary $w_n$ we are implicitly assuming that the presence of additive input does not have stabilizing or destabilizing effect on the mean-square exponential incremental stability of system (\ref{system_feedback_noise}).
%By requiring that (\ref{incremental_stable}) holds true for both the cases of with and without additive noise term we are 
%implicitly assuming that the presence of additive noise does not have stabilizing or destabilizing effect on the mean-square exponential incremental stability of system (\ref{system_feedback_noise}). 
\end{remark}

\noindent It is of interest to know what implication does incremental mean square exponential stability has on the boundedness of state space trajectories. Towards this we have following theorem
\begin{theorem}\label{theorem_secondmoment} Consider the following system
%two systems
%\begin{eqnarray}
%x_{n+1}=f(x_n)+\xi_n B k(x_n)+\xi_n B w_n=: F(x_n,\xi_n,w_n)\label{system_lemma_appendix}
%\end{eqnarray}
%and
\begin{eqnarray}
x_{n+1}=f(x_n)+\xi_n B k(x_n)+ \gamma_n\label{system_feedback_noise2}
\end{eqnarray}
If system (\ref{big_system}) is mean square exponentially incremental stable then
%\begin{enumerate}
%\item system (\ref{system_lemma_appendix}) is also mean square exponentially incremental stable i.e., there exists $K_1$ and $\beta_1$ such that
%\[E_{\xi_0^n}[\parallel x_{n+1}-y_{n+1}\parallel^2]\leq K_1 \beta_1^n \parallel x_0-y_0\parallel^2\]
%\item 
system (\ref{system_feedback_noise2}) is second moment bounded i.e.,
\[E_{\xi_0^n\gamma_0^n}[\parallel x_{n+1}\parallel^2]\leq \chi(x_0)<\infty\]
where $\chi$ is function of initial condition $x_0$ and the distribution of $\xi$ and constant $C$ bounding the variance of $\gamma$.
%\end{enumerate}
\end{theorem}
The proof of this theorem is provided in the Appendix \ref{appendix}.

%\begin{figure}[h*]
%\begin{center}
%\mbox{
%{\scalebox{.8}{\includegraphics[width=4in]{schematic}}}}
%\vspace{-0.2in}
%\caption{Schematic for feedback control over erasure channel}
%\label{schematic}
%\end{center}
%\end{figure}
\section{Motivating example}\label{section_motivation}
This section is introduced to motivate the problem set-up along with stability definitions and type of results we discover in this paper. We consider the following example of discrete-time Lorentz system with single input.
\begin{equation}
\begin{pmatrix}x_{n+1}\\y_{n+1}\end{pmatrix}=\begin{pmatrix}(1+\alpha \beta)(x_n+\sqrt{\alpha})-\beta (x_n+\sqrt{\alpha}) (y_n+\alpha) -\sqrt{\alpha}\\(1-\beta) (y_n+\alpha)+\beta (x_n+\sqrt{\alpha})^2-\alpha\end{pmatrix}+\xi_n\begin{pmatrix}1\\0\end{pmatrix} u_n+\gamma_n=\tilde F_s(z_n)+\xi_n B u_n+\gamma_n \label{lorentz_sys}
\end{equation}
where $z_n:=(x_n,y_n)$ are the states of the system, $u_n$ is the control input,  $\xi_n$ is the channel uncertainty, and $\gamma_n$ is Gaussian noise. The dynamics of the uncontrolled Lorentz system  for the parameter values of $\alpha=1.25$ and $\beta=0.75$ is shown in Fig. \ref{sim_section2}a. The dynamics  consists of a chaotic attractor and unstable equilibrium point at the origin. Trajectories starting from any two initial conditions on the attractor set exponentially diverge and the rate of expansion is given by the positive Lyapunov exponent. While the instability on the attractor set is captured by the positive Lyapunov exponent, the local instability near the unstable equilibrium point at the origin is captured by the eigenvalue of the linearization. The open loop dynamics of the Lorentz system is not atypical. In fact, the dynamics of one of the widely studied inverted pendulum with time-periodic forcing also consist of an unstable equilibrium point sitting inside a chaotic attractor.
%Our goal is to stabilize the unstable equilibrium point at the origin at the backdrop of channel erasure uncertainty.

There are two competing instabilities the controller needs to work against to stabilize the origin. The global instability on the attractor set and local instability at the origin. The limitation for  stabilization depends upon which of these two instabilities are more dominant. In fact, we show in Section \ref{section_simulation}, that the parameter values, $\alpha$ and $\beta$, can be chosen so that the global instability captured by the positive Lyapunov exponent is either larger or smaller than the local instability at the origin.  Our final objective is to provide analytical conditions that express limitations in terms of the instability of the open loop unstable dynamics as captured by the positive Lyapunov exponents and unstable eigenvalues at the origin. Since the Lyapunov exponent is the measure of expansion or contraction of nearby trajectories, this motivates us to consider  mean-square exponential incremental stability as the natural notion of stochastic stability for the closed loop system (Definition \ref{definition_stability}). The nonlinear function $\tilde F(z)$ in Eq. (\ref{lorentz_sys})  grows quadratically as $z\to \infty$ and hence it does not satisfy the Assumption \ref{assumption_1} of uniform bounded Jacobian.  In order to satisfy this assumption we modify the function $\tilde F$ as follows.
\begin{eqnarray}
z_{n+1}=\tilde F(z_n)G_1(z_n)+G_2(z_n)z_n=F(z_n)\label{mlorentz_sys}
\end{eqnarray}
where $G_1(z)=\frac{1}{2}(1+\tanh(k(M-\parallel z\parallel))$, and $G_2(z)=L(1+\tanh(k(\parallel z\parallel-M)))$. The functions $G_1$ and $G_2$ are chosen to approximate the sign function. The values of $k,M$, and $L$ are chosen large enough to ensure that the function $G_1(z)$ is one in the region $\parallel z\parallel\leq M$ containing the chaotic attractor of the system $z_{n+1}=\tilde F(z_n)$ and zero outside the region $\parallel z\parallel > M$. Similarly function $G_2(z)z$ is chosen to ensure that $G_2(z)z$ is zero in the region $\parallel z\parallel\leq M$ and is equal to $L z$ outside $\parallel z\parallel>M$. The values of $k,L$, and $M$ are chosen to be equal to $k=100, L=50$, and $ M=100$. With these values of $k,L$, and $M$ it can be shown the dynamics of system (\ref{mlorentz_sys}) is identical to that of (\ref{lorentz_sys}) in the region containing the chaotic attractor.

To discover the limitation results for the Lorentz system and further motivate the mean square exponential incremental stability definition, we consider the following coupled Lorentz system with one way coupling.
\begin{eqnarray}
w_{n+1}&=&F_m(w_n)\label{master}\\
z_{n+1}&=&F_s(z_n)+\xi_n B(u(z_n)-u(w_n))+\gamma_n\label{slave}
\end{eqnarray}
where $w_n\in \mathbb{R}^2$ is the state of another Lorentz system. With one way coupling from  equation (\ref{master}) to system (\ref{slave}) with the term $u(w_n)$, the system Eq. (\ref{slave}) is exactly in the same form as Eq. (\ref{system_feedback_noise}) in our problem formulation.  The goal is to mean square exponentially incrementally stabilize the system Eq. (\ref{slave}) and show that limitation for stabilization arise from the open loop dynamics, $z_{n+1}=F_s(z_n)$, and is independent of the additive forcing term $u(w_n)$ obtained using two different cases namely $F_m=F_s$ and $F_m\neq F_s$. For the case when $F_s=F_m$, the system equations (\ref{master})-(\ref{slave}) are in standard master-slave configuration, where mean square exponential incremental stability of (\ref{slave}) implies that the dynamics of slave system (\ref{slave}) is synchronized to that of master system (\ref{master}). This highlights the importance of the incremental stochastic stability definition for problems involving synchronization with more general network configuration and uncertainty in interaction.
 Since (\ref{master}) is another Lorentz system, the case $F_s\neq F_m$ is obtained by choosing two different set of parameter values for $(\alpha,\beta)$ from Eq. (\ref{lorentz_sys}). The parameters for system (\ref{master}) will be denoted by $(\alpha_m,\beta_m)$ and that for (\ref{slave}) system will be denoted by $(\alpha_s,\beta_s)$. In the following, we present simulation results to verify the limitations results  for system (\ref{slave}) for two different cases of $F_m=F_s$ and $F_m\neq F_s$. The main contribution of this paper is to provide rigorous proofs that help explain these simulation results.

\begin{figure}[h] \centering \subfigure[]
{\includegraphics[width=2.2in]{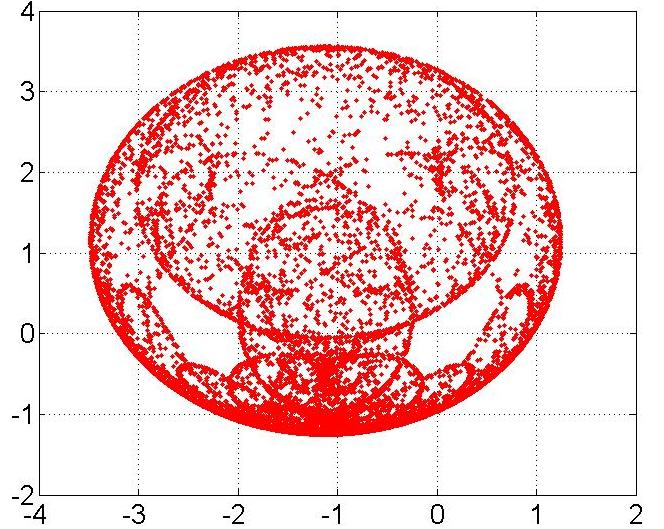}}\subfigure[]
{\includegraphics[width=2.2in]{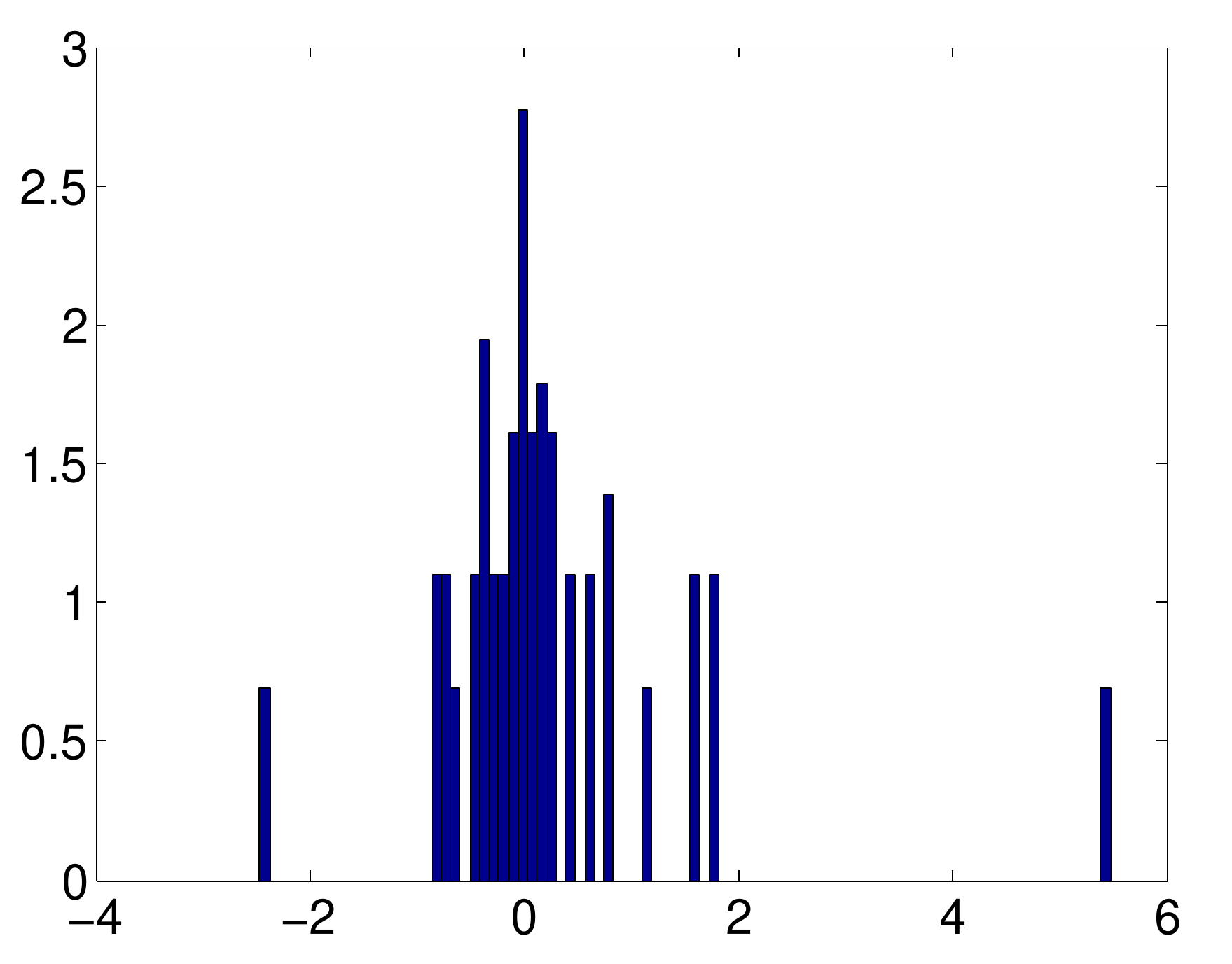}}\subfigure[]
{\includegraphics[width=2.2in]{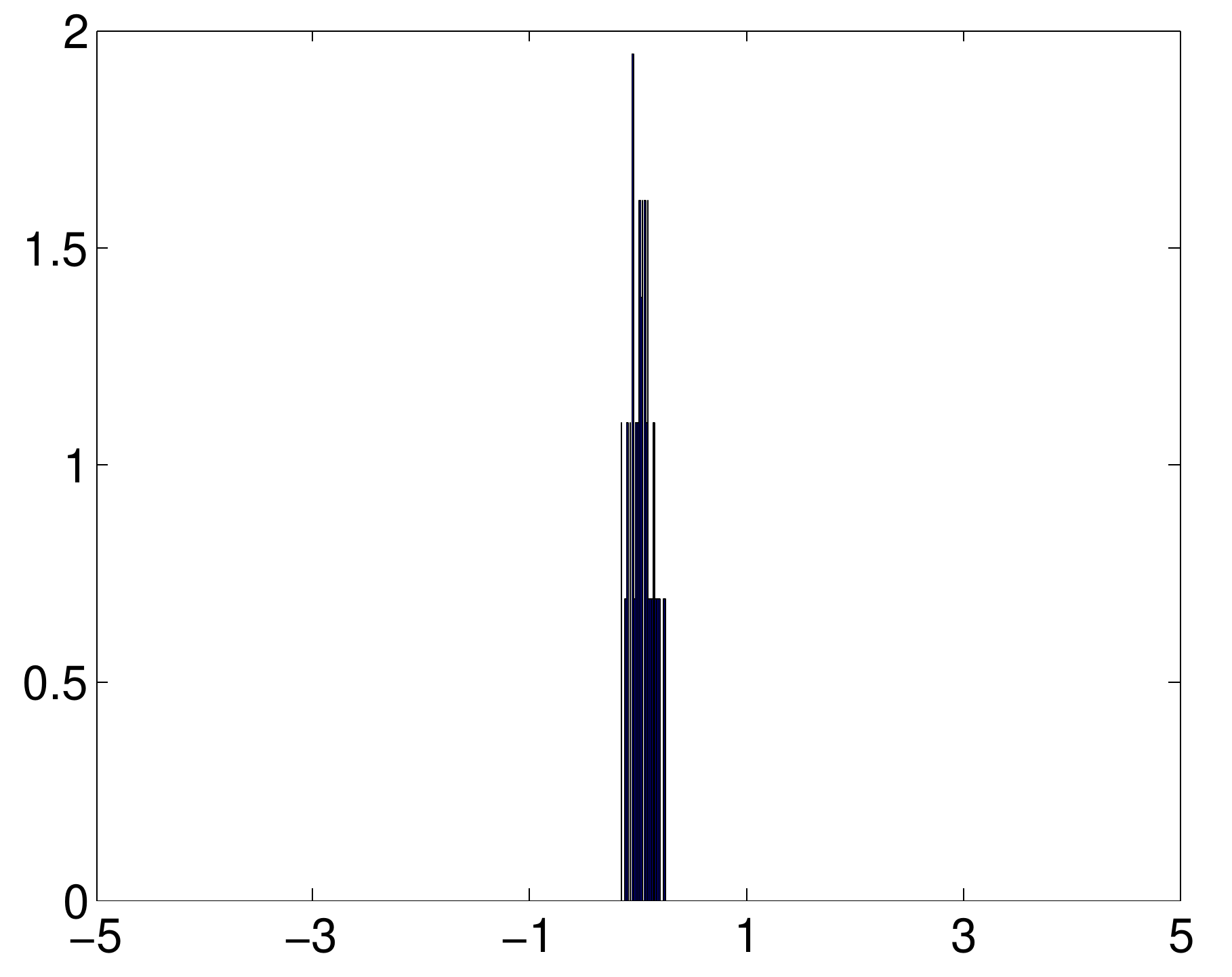}}
\caption{ a) Chaotic attractor  for uncontrolled Lorentz system for parameter value of $\alpha=1.25$ and $\beta=0.75$; b) Histogram for the error dynamics between two trajectories for non-erasure probability of $p=0.55$; c) Histogram for the error dynamics between two trajectories for non-erasure probability of $p=0.9$.}
\label{sim_section2}
\end{figure}
 \subsection{Simulations}
For the simulation we choose the parameter value of  $(\alpha_s,\beta_s)=(1.25,0.75)$. For these parameter values, the local instability at the equilibrium point is dominant over the global instability captured by positive Lyapunov exponents. The stabilizing feedback controller is designed to cancel the nonlinearity of the $x_{n+1}$ dynamics. We assume erasure channel uncertainty model for random variable $\xi_n$ with $Prob(\xi_n=1)=p$ and $Prob(\xi_n=0)=1-p$. Hence, $\mu=p$ and $\sigma^2=p(1-p)$ (Refer to Eq. (\ref{uncertain_statistics})) defines the statistics of the channel uncertainty. The random variable $\gamma_n$ is assumed to be zero mean and variance of $0.1^2$.

{\it Case $F_s\neq F_m$}: For this case the parameter value for system (\ref{master}) are chosen to be $(\alpha_m,\beta_m)=(2.25,0.29)$.
In Figs. \ref{sim_section2}b and \ref{sim_section2}c, we show the histogram for the error between the two trajectories for the system (\ref{slave}) for the non-erasure probability of $p=0.55$ and $p=0.9$ respectively. The histogram are obtained by averaging the error dynamics over $3000$ different realization of channel uncertainty random variable $\xi_n$.

  %The parameter values for the master system to obtain these simulation results in Figs. \ref{sim_section2}b \& \ref{sim_section2}c are taken to be $(\alpha_m,\beta_m)=(2.25,0.29)$.
  {\it Case $F_s=F_m$}: For this case the parameter values for both systems (\ref{master}) \& (\ref{slave}) are identical  i.e., $(\alpha_m,\beta_m)=(\alpha_s,\beta_s)=(1.25,0.75)$. For this case system (\ref{master})-(\ref{slave}) are in master-slave configuration.  The mean square exponentially incremental stability of slave system will imply that the slave system will synchronize its dynamics to that of master system.  The synchronization between master and slave system follows from the structure of the system equations (\ref{master}) and (\ref{slave}). It is clear that one of the trajectories of the slave system is the same as that of master system (i.e., $z_n=w_n \forall n$, if $z_0=w_0$). Hence, if the error between any two trajectories of the slave system approaches zero it implies that the slave system is tracking the master system. In Figs. \ref{sim_section2b}a \& \ref{sim_section2b}c, we plot the histogram for the error between the two trajectories for system (\ref{slave}).
%   for the case where master and slave system has same dynamics (i.e., $(\alpha_m,\beta_m)=(\alpha_s,\beta_s)=(1.25,0.75)$). With identical dynamics for the master and slave systems,  This highlights the significance of the mean square exponentially incremental stability notion for the problem of synchronization.
From Figs. \ref{sim_section2}(b,c) \& Figs. \ref{sim_section2b}(a,b), we notice that the wider support for the histogram for the non-erasure probability of $p=0.55$ is indicative of the fact that the the error dynamics have more variance around zero compared to the case of $p=0.9$.
%The large peak at $e=0$ is indicative of the fact the error dynamics is close to zero in both  cases.
%However, the wider support of the histogram corresponding to the smaller value of non-erasure probability, $p=0.5$, is indicative of the fact that the error dynamics have more variance around zero.

\begin{figure}[h] \centering \subfigure[]
{\includegraphics[width=2.2in]{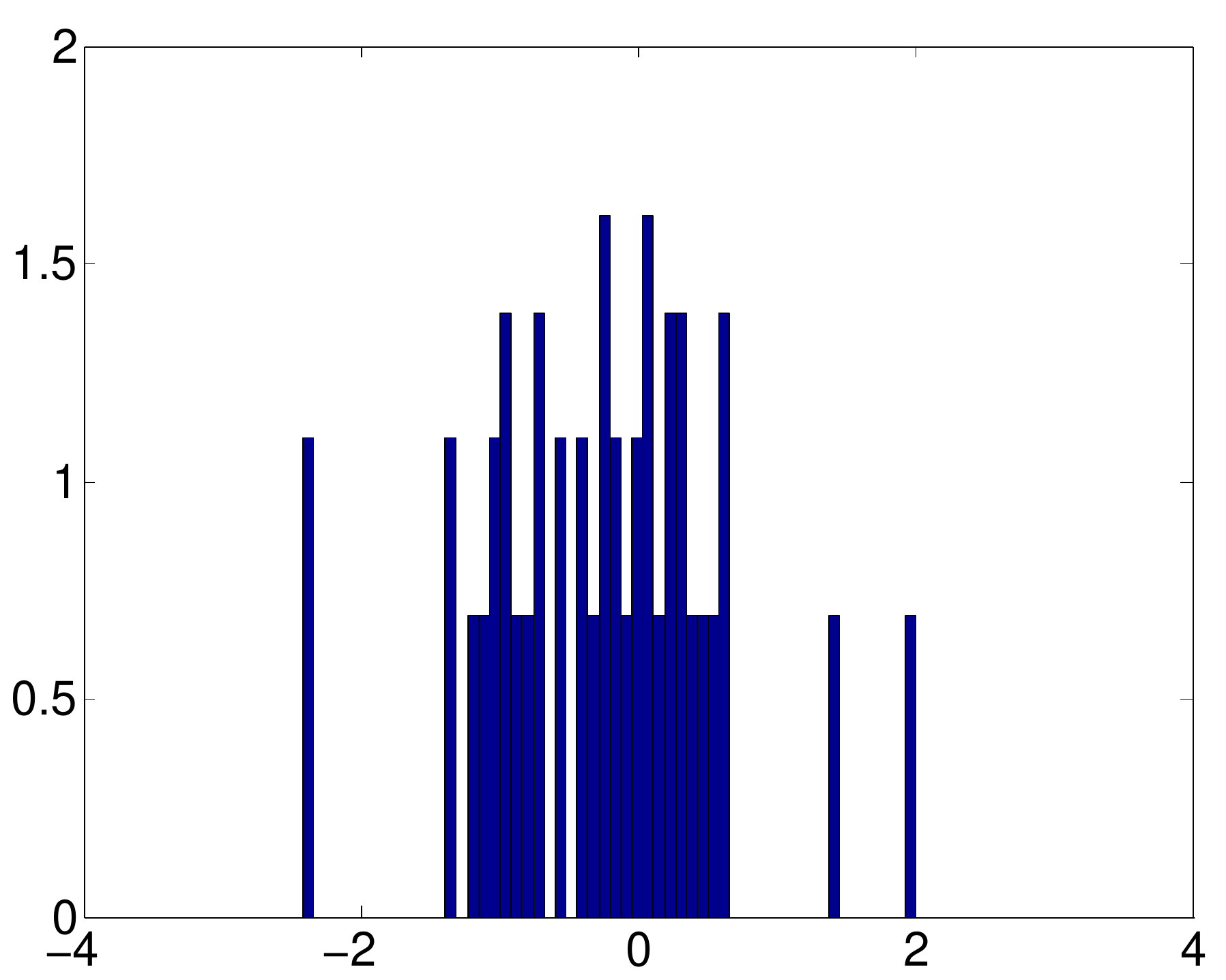}}\subfigure[]
{\includegraphics[width=2.2in]{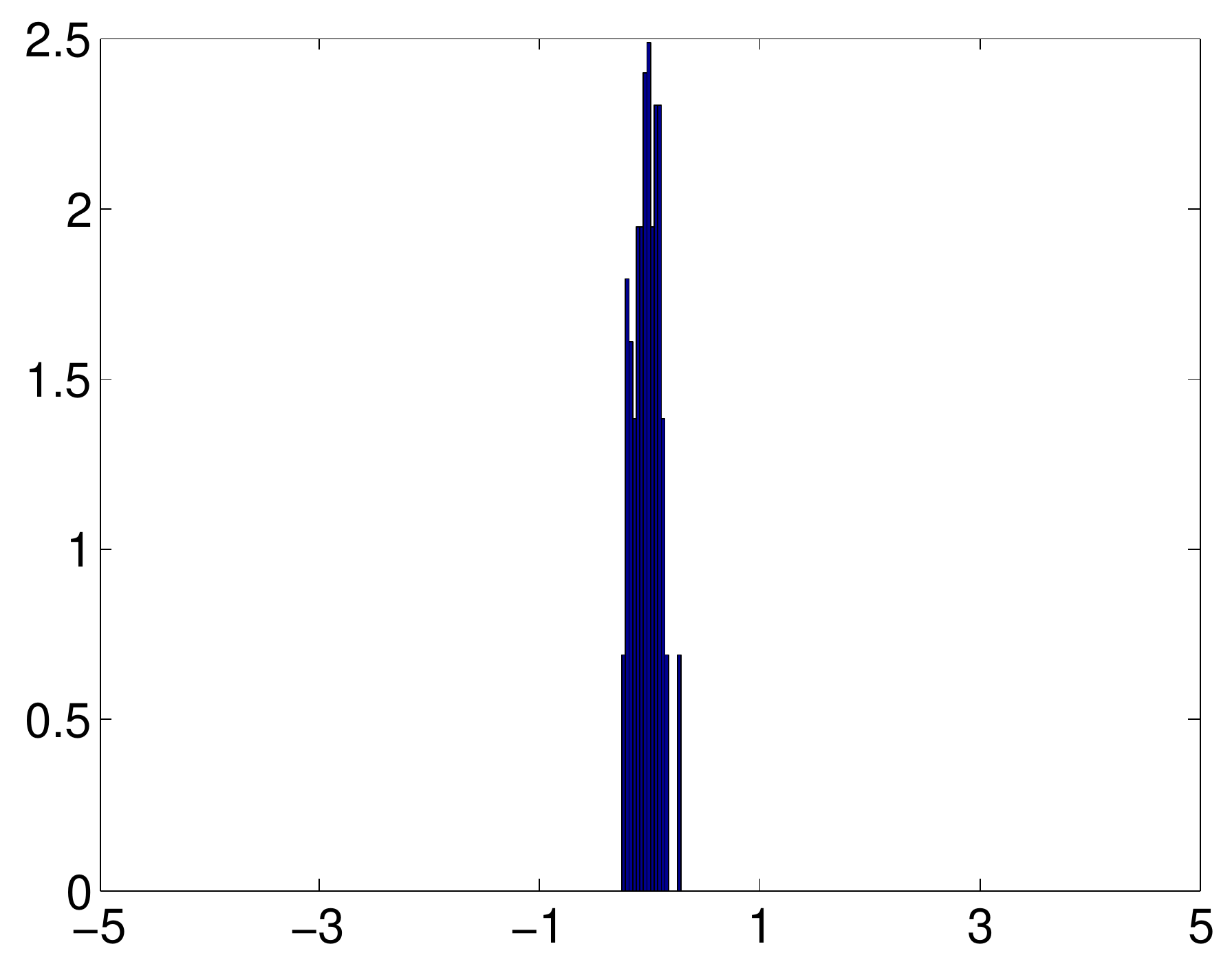}}\subfigure[]
{\includegraphics[width=2.2in]{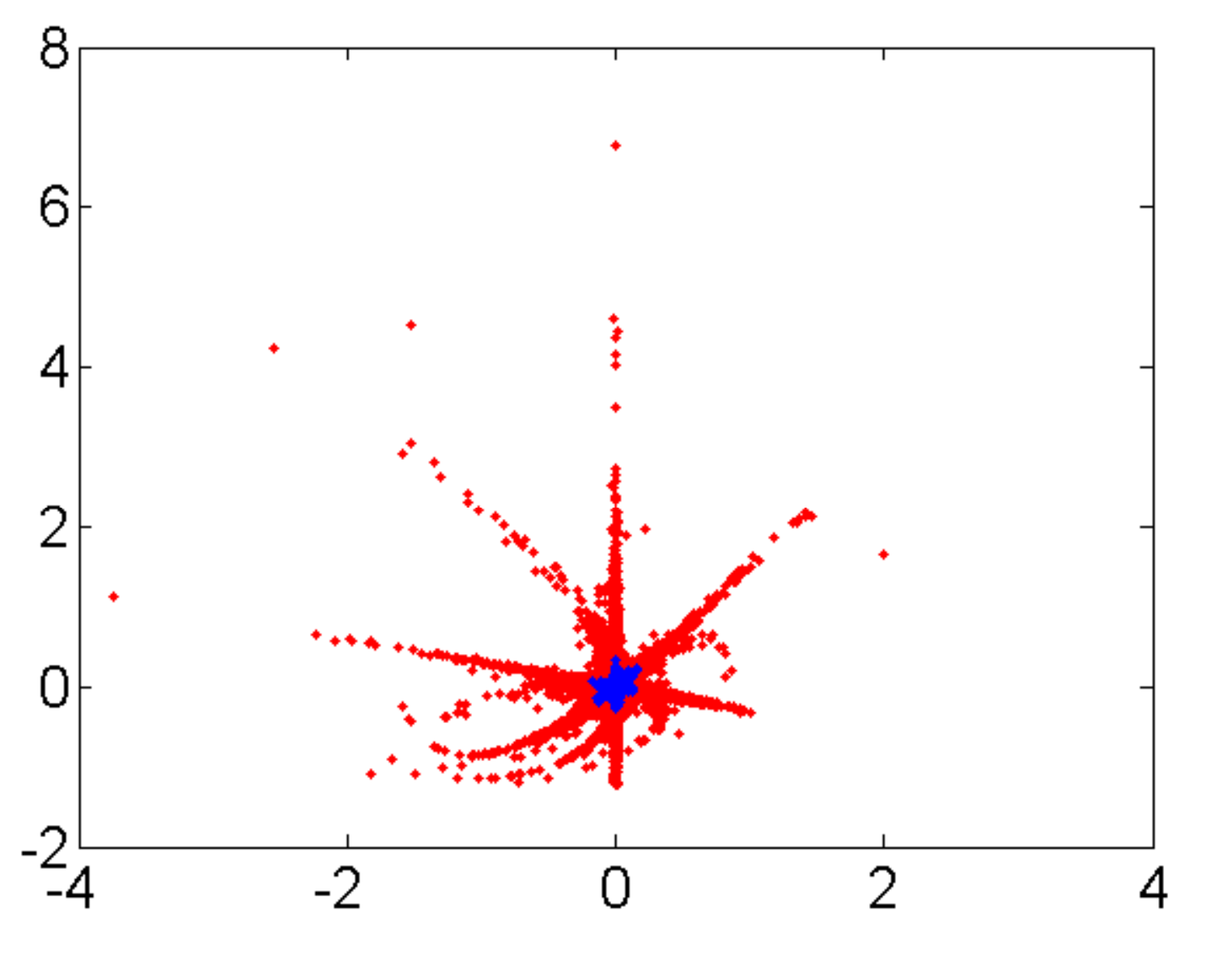}}
\caption{ a) Histogram for the error dynamics between two trajectories for non-erasure probability of $p=0.55$; b) Histogram for the error dynamics between two trajectories for non-erasure probability of $p=0.9$. c) Comparison of two attractors sets for $p=0.65$ (red) and $p=0.9$ (blue).}
\label{sim_section2b}
\end{figure}

%\begin{wrapfigure}[7]{l}{7cm}
%%[.5\columnsep]
%\vspace{-0.2in}
%\begin{figurehere}
%\begin{center}
%\vspace{-0.1in}
%{{\includegraphics[height=1.9in]{attractor_comparison_p065_p09_a001case1.pdf}}}
%\vspace{-0.3in}
%\caption{Comparison of two attractors sets for $p=0.65$ (red) and $p=0.9$ (blue)}
%\label{compare_attractor_eig}
%\end{center}
%\end{figurehere}
%\end{wrapfigure}

From the simulation results presented in Figs. (\ref{sim_section2}) (b,c) \& (\ref{sim_section2b}) (a,b), we conclude that the irrespective of whether $F_m=F_s$ or $F_m\neq F_s$, the  error dynamics has larger variance for non-erasure probability below $p=0.55$. During this time, the error dynamics is away from zero, loosely speaking, the system trajectories wander on the chaotic attractor. This behavior is sensitive to the presence of additive noise. In particular,  additive noise  will prevent the error dynamics from converging to zero, thereby manifesting the mean-square instability of the error dynamics in the error trajectories. In Fig. \ref{sim_section2b}c, we plot the attractor set for the slave dynamics (i.e., Eq. (\ref{slave})) without the coupling from the master system (i.e., $u(w_n)=0$) for two different values of non-erasure probabilities of $p=0.65$ and $p=0.9$.
Comparison of the two attractor sets reveals that attractor set for $p=0.65$ is chaotic, while for $p=0.9$, the attractor is tame with support close to the origin.

The objective of this paper is to provide a rigorous framework that will allow us to distinguish the mean-square stable and unstable behaviors of the error dynamics. In fact, using the main results of this paper in section \ref{section_main} it can be shown that  for the parameter value of $\alpha=1.25$ and $\beta=0.75$ used for the simulation, the critical non-erasure probability below, which the error dynamics for the Lorentz system (\ref{slave}), is mean-square unstable is $p^*=0.63$ (or equivalently $\rho^*=\frac{\mu^2}{\mu^2+\sigma^2}=0.63$). We will revisit this example with more simulation results later in the simulations Section \ref{section_simulation}. We will show that  for parameter values of $(\alpha_s,\beta_s)=(2.25,0.29)$, where the global instability captured by the positive Lyapunov exponent is dominant over the local instability, the limitations and hence the critical non-erasure probability is  determined based on global instability.

%To demonstrate the effect of small additive noise, we show the comparison of two attractor sets obtained for the erasure probabilities of $p=0.5$ and $p=0.9$ in Fig. \ref{compare_attractor_eig} with a small uniform random variable, $\varpi_n\in[0,0.01]$, added to state, $x_{n+1}$. Comparison of the two attractor sets reveals that attractor set for $p=0.5$ is chaotic, while for $p=0.9$, the attractor is tame with support close to the origin.  The main results of this paper are prove without any additive noise term because the mathematical machinery required to study the nonlinear system with additive noise on compact space is beyond the scope of this paper.  The objective of this paper is to provide a rigorous framework that will allow us to distinguish the mean-square stable and unstable behaviors of the error dynamics as shown in Figs. \ref{sim_section2}b, \ref{sim_section2}c, and \ref{compare_attractor_eig}. In fact, we prove in Section \ref{section_main} that for the parameter value of $\alpha=1.25$ and $\beta=0.75$ used for the simulation, the critical non-erasure probability below, which the error dynamics for the Lorentz system, is mean-square unstable is $p^*=0.63$. We will revisit this example with more simulation results later in the simulations Section \ref{section_simulation}.
%}}

\section{Fundamental limitations for incremental stabilization}\label{limitation_section}
The main result of this paper proves the mean-square exponential incremental stabilization of the networked system requires certain minimal QoS from the network. Next, we outline the various steps involved in the proof of this main result.
%
%Given the complicated nature of the proof for the limitation on mean square stabilization, we  outline various steps that are involved in the proof of the main result.
\begin{enumerate}
\item First, we show a necessary condition for the mean-square incremental exponential stability of (\ref{big_system}) can be expressed  in terms of the mean-square exponential stability of its linearization  along the system trajectory (Theorem \ref{linearization}).
%
%\item We next define the relaxation of system equation (\ref{system}). Doing so we derive a necessary condition for the second moment stability of relaxed system thereby providing  necessary condition for the second moment stability of the actual system.
\item  Next, the necessary condition for the mean-square exponential stability for the linearized system is based on the Lyapunov analysis.
%\item We use the Lyapunov-based analysis to derive the optimal control that minimize the second moment for the linearized system.

\item Finally, the optimal control derived using Lyapunov analysis is used to prove the main results on the mean-square exponential incremental stability.

\end{enumerate}
%For notational  convenience, we write the system (\ref{system}) as follows:
%\vspace{-0.1in}
%\begin{eqnarray}
%x_{n+1}=f(x_n)+\xi_n Bk(x_n)=: F(x_n,\xi_n)\label{system_with_feedback}
%\end{eqnarray}
%\vspace{-0.1in}
%The trivial consequence of  definition \ref{mss_condition} is that for the system (\ref{system_with_feedback}) to be mean square exponential stable it is necessary that $x=0$ is an equilibrium point of (\ref{system_with_feedback}).

\begin{theorem}\label{linearization} Consider the following system dynamics.
\begin{eqnarray}
\eta_{n+1}&=&\left(\frac{\partial f}{\partial x}(x_n)+ \xi_n B \frac{\partial k}{\partial x}(x_n)\right)\eta_n\label{linearized_system}\\
x_{n+1}&=&f(x_n)\label{nominal}.
\end{eqnarray}
Let system (\ref{big_system}) be mean-square exponential incremental stable then, the linearized dynamics (\ref{linearized_system}) is mean-square exponential stable, i.e., there exist positive constants $K_2<\infty$ and $\beta_2<1$, such that,
\[E_{\xi_0^n}\left[\parallel \eta_{n+1}\parallel^2\right]\leq K_2\beta_2^n \parallel \eta_0\parallel^2\;\;\;\;\forall n\geq 0,\]
for Lebesgue almost all initial conditions $x_0\in \mathbb{R}^N$ and $x_0=0$.
\end{theorem}
We provide the proof of this Theorem in Appendix \ref{appendix}.

%\noindent For future reference we write the  linearized system (\ref{linearized_system}) from Theorem \ref{linearization} as follows:
%\vspace{-0.1in}
%\begin{subequations}
%\begin{equation}
%x_{n+1}=f(x_n)+ B k(x_n)
%\end{equation}
%\begin{equation}
%\eta_{n+1}=\left(\frac{\partial f}{\partial x}(x_n)+\xi_n B \frac{\partial k}{\partial x}(x_n)\right)\eta_n=: {\cal A}(x_n,\xi_n)\eta_n \end{equation}\label{linearized_system_x}
%\end{subequations}
%\vspace{-0.1in}
%Theorem \ref{linearized_system}  motivates us to consider the following definition of mean square exponential stability of (\ref{linearized_system_x})
%\begin{definition}[Mean square exponential stability of linearized system]\label{mss_linearized} The linearized system (\ref{linearized_system_x}) is said to be mean square exponential stable if there exists positive constant $K<\infty$ and $\beta<1$ such that
%\vspace{-0.1in}
%\[E_{\xi_0^n}[\parallel \eta_{n+1}\parallel^2]\leq K\beta^n (\parallel x_0 \parallel+ \parallel \eta_0 \parallel)^2 \]
%\vspace{-0.1in}
%for Lebesgue almost all initial conditions $x_0\in X$ and for $x_0=0$.
%\end{definition}

Now, we provide the Lyapunov-based necessary condition for the mean-square exponential stability of a linearized system,
\begin{subequations}
\begin{equation}
\eta_{n+1}=\left(\frac{\partial f}{\partial x}(x_n)+\xi_n B \frac{\partial k}{\partial x}(x_n)\right)\eta_n\label{linearized_a}
\end{equation}
\begin{equation}
x_{n+1}=f(x_n).\label{linearized_b}
\end{equation}\label{linear}
\end{subequations}
\begin{theorem}\label{Lya_theorem}A necessary condition for the linearized system (\ref{linear}), with controller mapping $k$ satisfying Assumption \ref{assumption_controller}, to be mean-square exponentially stable is  there exist positive constants, $\alpha_1$ and $\alpha_2$, and a matrix function of $x$, $P(x)$, such that, $\alpha_1 I\leq P(x)\leq \alpha_2 I$ and
\begin{eqnarray}
E_{\xi_\ell}\left[ {\cal A}^{'}(x_\ell,\xi_\ell)P(x_{\ell+1}){\cal A}(x_\ell,\xi_\ell)\right]< P(x_\ell), \label{stability_cond}
\end{eqnarray}
for almost all with respect to the Lebesgue measure, $x_0\in \mathbb{R}^N$ and $x_0=0$, where ${\cal A}(x_\ell,\xi_\ell):=\frac{\partial f}{\partial x}(x_\ell)+\xi_\ell B \frac{\partial k}{\partial x}(x_\ell)$ and $x_{\ell+1}=f(x_\ell)$.

% such that $\alpha_2(x)\to \infty$ as $\parallel x \parallel \to \infty$.
\end{theorem}
\begin{proof}
Consider the following construction of $P(x)$.
\begin{eqnarray}
P(x_\ell):=\sum_{n=\ell}^{\infty}E_{\xi_{\ell}^n}\left[
\left(\prod_{k=\ell}^n {\cal A}(x_k,\xi_{k})\right)^{'}\left(\prod_{k=\ell}^n {\cal A}(x_k,\xi_k)\right)\right]. \label{formula}
\end{eqnarray}
From the above construction of $P(x_\ell)$ and making use of Assumption \ref{assumption_controller} it follows $P(x_\ell)$ satisfies the inequality (\ref{stability_cond}). Furthermore, since system (\ref{linear}) is assumed mean-square exponentially stable, we know there exist positive constants $K<\infty$ and $\beta<1$, such that
\[E_{\xi_0^{\ell}}\left[\parallel {\cal A}(x_\ell,\xi_\ell)\cdots {\cal A}(x_0,\xi_0)\parallel^2 \right]\leq K_2 \beta_2^{ \ell}.\]
Hence, there exists a positive constant, $\alpha_2$, such that $P(x)\leq \alpha_2 I$. The lower bound on $P(x)$ and the existence of $\alpha_1$ follow from the construction of $P(x)$ and Assumption \ref{assumption_controller}.

%With the above definition of $P(x_\ell)$ and using the fact $E_{\xi_n}[({\cal A}(x_k,\xi_k))^{'}{\cal A}(x_k,\xi_k)]> 0$, since ${\cal A}(x,\xi=0)=\frac{\partial f}{\partial x}$ is assumed to be invertible (Assumption \ref{assumption_1}), it follows $P(x_\ell)$ satisfies the inequality (\ref{stability_cond}).  Furthermore, since system (\ref{linear}) is assumed mean-square exponentially stable, we know there exist positive constants $K<\infty$ and $\beta<1$, such that
%\[E_{\xi_0^{\ell}}\left[\parallel {\cal A}(x_\ell,\xi_\ell)\cdots {\cal A}(x_0,\xi_0)\parallel^2 \right]\leq K_2 \beta_2^{ \ell}.\]
%Hence, from the construction of $P(x_\ell)$, it follows there exists a positive constant, $\alpha_2$, such that $P(x)\leq \alpha_2 I$.
%The lower bound on $P(x)$ and the existence of $\alpha_1$ follow from the construction of $P(x)$ and
%the assumption that ${\cal A}(x,\xi=0)$ is invertible.
\end{proof}

\begin{definition}[Matrix Lyapunov function]\label{def_matrix_Lya} We refer to the matrix function $P(x)$ satisfying the necessary condition (\ref{stability_cond})  of Theorem \ref{Lya_theorem} as a matrix Lyapunov function.
\end{definition}
We now use the matrix Lyapunov function to derive the necessary condition for mean-square exponentail stability.
%
%optimal control that minimize the second moment for system (\ref{linearized_system_x}).
\begin{theorem}\label{theorem_support} The necessary condition for the mean-square exponential stability of a  linearized system (\ref{linear}) is given by
\begin{eqnarray}
A^{'}(x_n)Q_0(x_{n+1})A(x_n)-  \rho A^{'}(x_n) Q_0(x_{n+1}) B\left(B^{'} Q_0(x_{n+1})B\right)^{-1} B^{'}Q_0(x_{n+1}) A(x_n)<Q_0(x_n),\label{necc_condition11}
\end{eqnarray}
%\begin{eqnarray}
%(1-p)\left(I+B^{'} Q_0(x)B\right)<I,\label{necc_condition11}
%\end{eqnarray}
for almost all, with respect to the Lebesgue measure initial condition, $x_0\in \mathbb{R}^N$ and for $x_0=0$ and where $\rho=\frac{\mu^2}{\mu^2+\sigma^2}$. The matrix $Q_0(x)$ is the solution of the following Riccati-like equation.
\begin{eqnarray}
&A^{'}(x_n)Q_0(x_{n+1})A(x_n)-  A^{'}(x_n) Q_0(x_{n+1}) B\left( I+B^{'} Q_0(x_{n+1})B\right)^{-1} B^{'}Q_0(x_{n+1}) A(x_n)+R(x_n)\nonumber\\&=Q_0(x_n),\label{riccati}
\end{eqnarray}
for some $R(x_n)\geq 0$, and where $A(x_n):=\frac{\partial f}{\partial x}(x_n)$, $x_{n+1}=f(x_n)$.
% Hessian of the optimal cost function $V(x)$ corresponding to the minimum energy optimal control problem for the system $x_{n+1}=f(x_n)+Bu_n$ (i.e., $Q_0(x)=\frac{\partial^2 V}{\partial x^2}(x)$) and satisfies following equation
%\vspace{-0.1in}
%\begin{eqnarray}
%A^{'}(x_n) Q_0(x_{n+1})A(x_n)- A^{'}(x_n) Q_0(x_{n+1})B\left( I+ B^{'} Q_0(x_{n+1})B\right)^{-1} B^{'} Q_0(x_{n+1})A(x_n)=Q_0(x_n)\label{riccati_P}
%\end{eqnarray}
%\vspace{-0.1in}
%where $A(x_n):=\frac{\partial f}{\partial x}(x_n)$ and $ x_{n+1}=f(x_n)+Bu_n$.
% The  matrix function $Q_0(x)$ can be obtained as the solution to the following Riccati-like equation corresponding to the minimum energy optimal control problem (Appendix \ref{appendix})
%\begin{eqnarray}
%A^{'}(x_n) Q_0(x_{n+1})A(x_n)- A^{'}(x_n) Q_0(x_{n+1})B\left( I+ B^{'} Q_0(x_{n+1})B\right)^{-1} B^{'} Q_0(x_{n+1})A(x_n)=Q_0(x_n)\label{riccati_P}
%\end{eqnarray}
%Furthermore the matrix function $ Q_0(x)$ will qualify as the valid Lyapunov function (Definition \ref{def_matrix_Lya}) provided
%\[(1-p)\left(I+B^{'} Q_0(x_{n+1})B\right)<I\]
\end{theorem}
\begin{proof} Using the results from Theorem \ref{Lya_theorem}, we know  the necessary condition for the mean-square exponential stability of (\ref{linear}) can be expressed in terms of the existence of a matrix, Lyapunov function $P(x)$, such that the following inequalities are satisfied, $\alpha_1 I\leq P(x)\leq \alpha_2 I$.
%Let $P(x_n)$ be the matrix Lyapunov function satisfying the condition of the Theorem \ref{Lya_theorem}. To derive the optimal control for the derivative map dynamics (\ref{linearized_system_x}b), we write the control Lyapunov inequality as follows:
\begin{eqnarray}
E_{\xi_n}\left[ \left(A(x_n)+\xi_n B \frac{\partial k}{\partial x}(x_n)\right)^{'}P(x_{n+1})\left(A(x_n)+\xi_n B \frac{\partial k}{\partial x}(x_n)\right)\right]<P(x_n).\label{Expec}
\end{eqnarray}
Taking expectation w.r.t. $\xi_n$,  using the fact that $x_n$ is independent of $\xi_n$, and minimizing the trace of the left-hand side of (\ref{Expec}) with respect to $\frac{\partial k}{\partial x}$, we obtain the following expression for the optimal control, $\frac{\partial k}{\partial x}$, in terms of  $P(x_n)$,
\begin{eqnarray}
\frac{\partial k}{\partial x}(x_n)=-\frac{\mu}{\sigma^2+\mu^2}\left(B^{'} P(x_{n+1})B\right)^{-1}B^{'} P(x_{n+1})A(x_n).\label{optimal_control_tangent}
\end{eqnarray}
Substituting (\ref{optimal_control_tangent}) in (\ref{Expec}), we obtain the following necessary condition for the mean-square exponential stability,
\begin{eqnarray}
A^{'}(x_n)P(x_{n+1})A(x_n)-  \rho A^{'}(x_n) P(x_{n+1}) B\left(B^{'} P(x_{n+1})B\right)^{-1} B^{'}P(x_{n+1}) A(x_n) <P(x_n)\label{zz}.
\end{eqnarray}
where $\rho:=\frac{\mu^2}{\mu^2+\sigma^2}$.
It is important to notice  the above inequality is independent of scaling, i.e, if $P(x_n)$ satisfies the above inequality, then $\gamma P(x_n)$ also satisfies above inequality for any positive constant, $\gamma>0$.
By defining $\Delta_P:=B^{'} P(x_{n+1})B\frac{\sigma^2}{\mu^2}$, we write the above inequality as follows:
\begin{eqnarray}
A^{'}(x_n)P(x_{n+1})A(x_n)-  A^{'}(x_n) P(x_{n+1}) B\left(\Delta_P+B^{'} P(x_{n+1})B\right)^{-1} B^{'}P(x_{n+1}) A(x_n) <P(x_n). \label{inequality}
\end{eqnarray}
Since $P(x_n)$ is matrix Lyapunov function, and hence, bounded below, we know there exists a positive constant $\Delta>0$ such that $\Delta I \geq \Delta_P$. Hence, (\ref{inequality}) implies $P(x_n)$ satisfies the following inequality,
\begin{eqnarray}
A^{'}(x_n)P(x_{n+1})A(x_n)-  A^{'}(x_n) P(x_{n+1}) B\left(\Delta I+B^{'} P(x_{n+1})B\right)^{-1} B^{'}P(x_{n+1}) A(x_n) <P(x_n). \label{inequality_1}
\end{eqnarray}
Now, since (\ref{inequality_1}) is independent of positive scaling, we obtain from (\ref{inequality_1})
\begin{eqnarray}
A^{'}(x_n)Q_0(x_{n+1})A(x_n)-A^{'}(x_n) Q_0(x_{n+1}) B\left( I+B^{'} Q_0(x_{n+1})B\right)^{-1} B^{'}Q_0(x_{n+1}) A(x_n) <Q_0(x_n), \label{riccati_Q0_inequality}
\end{eqnarray}
where $Q_0(x_n):=\frac{1}{\Delta} P(x_n)$. Furthermore (\ref{riccati_Q0_inequality}) implies existence of $R(x_n)\geq 0$, such that following Riccati-like equality is true
\begin{eqnarray}
&A^{'}(x_n)Q_0(x_{n+1})A(x_n)-A^{'}(x_n) Q_0(x_{n+1}) B\left( I+B^{'} Q_0(x_{n+1})B\right)^{-1} B^{'}Q_0(x_{n+1}) A(x_n) \nonumber\\&+R(x_n)=Q_0(x_n). \label{riccati_Q0}
\end{eqnarray}
The necessary condition (\ref{necc_condition11}) then follows from (\ref{zz}). Using the definition, $P(x_n)=\Delta Q_0(x_n)$. The Riccati-like equation (\ref{riccati_Q0}) resembles the Riccati equation obtained in the solution of the linear quadratic regulator problem for the linear time varying (LTV) system \cite{kwakernaak} with the difference that various matrices in (\ref{riccati_Q0}) are parameterized by a state trajectory $\{x_n\}$ of system, $x_{n+1}=f(x_n)$, instead of time. The matrix function $Q_0(x_n)$ is also bounded both above and below. It follows from the uniformly, completely controllability assumption on pair $(f(x),B)$ (Assumption \ref{assumption_2}).
\end{proof}

\section{Main results}\label{section_main}
In this section, we use the result from Theorem \ref{theorem_support} to derive conditions for mean-square exponential incremental stability under various assumptions on system dynamics.
\subsection{Linear system}
\begin{theorem}\label{linear_systems_cond} For the linear time invariant system (i.e., $f(x)=Ax$) with all eigenvalues having an absolute value greater than one, a necessary condition for the mean-square exponential incremental stability with controller satisfying Assumption \ref{assumption_controller}
\begin{itemize}
\item for number of control inputs, $1\leq d<N$, is given by
\begin{eqnarray}
\bar \rho^d\prod_{k=1}^N |\lambda_k|^2<1.\label{linear_cond}
\end{eqnarray}
\item for the number of control inputs equal to the number of states, i.e., $d=N$ and $B$ invertible, is given by %\vspace{-0.15in}
\begin{eqnarray}
\bar \rho |\lambda_{max}|^2<1,\label{linear_cond_N}
\end{eqnarray}
\end{itemize}
where, $\bar \rho=1-\rho$, $\rho=\frac{\mu^2}{\mu^2+\sigma^2}$, $\lambda_k$ for $k=1,\ldots,N$ and $\lambda_{max}$ are the unstable eigenvalues and maximum eigenvalue of matrix $A$, respectively.
\end{theorem}
\begin{proof}
{\it For $1\leq d<N$:}
A necessary condition (\ref{necc_condition11}) for mean-square exponential incremental stability  from Theorem \ref{theorem_support} for the special case of a linear time invariant system can be written as,
\begin{eqnarray}
A'Q_0A-\rho A'Q_0 B(B'Q_0 B)^{-1}B'Q_0A'<Q_0,\label{eqn_linear}
\end{eqnarray}
Taking determinants on both  sides of  the above equation and using the matrix determinant formula, i.e., $\det(I_N-\rho B(B'Q_0B)^{-1}B'Q_0)=\det (I_d-\rho (B'Q_0B)^{-1}B'Q_0B)$, we obtain
\[\det (A'Q_0)\det(\bar \rho I_d)\det(A)<\det (Q_0),\]
where $I_d$ in $d\times d$ identity matrix. Simplifying above inequality, we obtain,
\begin{eqnarray}
\bar \rho ^d\det(A)^2 =\bar \rho ^d\prod_{k=1}^N |\lambda_k|^2.\label{product_eig_lin}
\end{eqnarray}
{\it $N$ input case}: For the $N$ input case, matrix $B$ is invertible and a necessary condition for the mean-square exponential incremental stability (\ref{necc_condition11}) from Theorem \ref{theorem_support}, and therefore reduces to
\begin{eqnarray}\bar \rho A'Q_0A<Q_0.\label{cond_N}
\end{eqnarray}
%and with $Q_0$ satisfying the Riccati equation $A'Q_0A-A'Q_0B(I+B'Q_0B)^{-1}B'Q_0A=Q_0$,
A necessary condition for satisfying inequality (\ref{cond_N}) is given by
\[\bar \rho|\lambda_{\max}|^2<1.\].
\end{proof}
\begin{remark} Careful examination of the proofs for Theorems \ref{linearization} and  \ref{theorem_support}, for the special case of LTI systems, reveals  the conditions in Theorem \ref{linear_systems_cond}  are also sufficient for the mean-square exponential incremental stability of the LTI system for $d=1$ and $d=N$. Furthermore, it is not difficult to prove  the LTI system is mean-square exponential incremental stable, if and only if, the origin of the system is mean-square exponential stable. Hence, the results of Theorem \ref{linear_systems_cond} will also provide necessary and sufficient conditions for the mean-square exponential stability of LTI systems over erasure channels for $d=1$ and $d=N$. The results from  Theorem \ref{linear_systems_cond} are consistent and in agreement with the known results for the control of LTI systems over erasure channels for $d=1$ and $d=N$ cases \cite{scl04}.
\end{remark}

%\subsection{Nonlinear system on unbounded state space}
%Theorem \ref{result_nonlinear}  provides a necessary condition for mean-square exponential incremental stability for a general nonlinear system.
%\begin{theorem}\label{result_nonlinear} Consider  system (\ref{system_feedback}) with system mapping, $f$, satisfying Assumptions \ref{assumption_1} and \ref{assumption_2}, and the state space, $X$, not necessarily bounded. Then, a necessary condition for the mean-square exponential incremental stability with controller satisfying Assumption \ref{assumption_controller}   is given by \vspace{-0.15in}
%\begin{eqnarray}
%A'(x_n)Q_0(x_{n+1})A(x_n)-pA'(x_n)Q_0(x_{n+1})B(B'Q_0(x_{n+1})B)^{-1}B'Q_0(x_{n+1})A(x_n)<Q_0(x_n),\label{necessary_unbounded}
%\end{eqnarray}
%for Lebesgue almost all initial conditions, $x_0\in X$ and $x_0=0$. Matrix $Q_0(x)$ satisfies the following Riccati-like equation, \vspace{-0.15in}
%\[A^{'}(x_n)Q_0(x_{n+1})A(x_n)-A^{'}(x_n) Q_0(x_{n+1}) B(I+B^{'} Q_0(x_{n+1})B)^{-1} B^{'}Q_0(x_{n+1}) A(x_n) +R(x_n)=Q_0(x_n),\]
%for some $R(x_n)\geq 0$. The matrix $Q_0(x_n)$ is bounded both above and below with $x_{n+1}=f(x_n)$ and $A(x_n):=\frac{\partial f}{\partial x}(x_n)$.
%\end{theorem}
%\begin{proof} The proof follows by combining the results of Theorems \ref{linearization}, \ref{Lya_theorem}, and \ref{theorem_support}.
%\end{proof}
%
%
%
%

\subsection{Nonlinear system with ergodicity assumption}
The main theorem of this section provides a stability condition for nonlinear systems under some ergodicity assumption on system dynamics. In particular, we assume that the uncontrolled system has unique invariant measure  (Definition \ref{ergodic_measure}) and associated positive Lyapunov exponents (Definition \ref{Lyapunov_exponents}). Lyapunov function-based argument and theorems exists to ensure existence of invariant measure on unbounded state space \cite{Lasota} \footnote{For system with compact state space the existence of invariant measure is guaranteed under the continuity assumption on the system mapping. All the results from this section will apply to the compact state space case,  but we prefer to address the noncompact case as it more easily connected to the existing results. }.  Existence of invariant measure guarantee that the system has well defined steady state where the system dynamics eventually settle down. We start with the following definition of physical invariant measure.

%evolving on compact state space. The compactness assumption on the
%state space is  made to ensure existence of at least one invariant measure. The existence of invariant measure can also be guaranteed without making the compactness assumption by using the Lyapunov-based argument \cite{Lasota} or results from Markov chain theory \cite{Meyn}. However the compactness assumption allows us to use results and tools from ergodic theory of dynamical systems. In particular, with the
%existence of an invariant measure and the associated ergodic partition, quantities, like Lyapunov
%exponents, are well-defined for the system. Furthermore, the compactness assumption allows us to use the results from Ruelle's inequality, Theorem \ref{theorem_ruelle_inequality}, relating measure theoretic entropy and the positive Lyapunov exponents of the system.  To prove the main theorem of this section, we require the following definitions and results from the ergodic theory of dynamical systems.
%

\begin{definition}[Physical invariant measure]\label{ergodic_measure} A probability measure $\mu$ defined on $\mathbb{R}^N$, is said to be invariant for $x_{n+1}=f(x_n)$ if $\mu(B)=\mu(f^{-1}(B))$ for all sets $B\in {\cal B}(\mathbb{R}^N)$ (where $f^{-1}(B)$ is the inverse image of set $B$ and ${\cal B}(\mathbb{R}^N)$ is the Borel-$\sigma$ algebra on $\mathbb{R}^N$). An invariant probability measure is said to be ergodic, if any $f$-invariant set, $A$, i.e., $f^{-1}(A)=A$, has $\mu$ measure equal to one or zero. The ergodic invariant measure is said to be physical, if 
\begin{eqnarray}\lim_{n\to \infty}\frac{1}{n}\sum_{k=0}^n \varphi(x_k)=\int_{\mathbb{R}^N} \varphi(x)d\mu(x),\label{ec1}
\end{eqnarray} for positive Lebesgue measure set of initial condition $x_0\in \mathbb{R}^N$ and for all continuous function $\varphi:\mathbb{R}^N\to \mathbb{R}$. The physical measure, $\mu$, is said to be unique if (\ref{ec1}) holds true for all Lebesgue measure initial condition $x_0\in \mathbb{R}^N$ and all continuous function $\varphi$.
\end{definition}
Definition \ref{ergodic_measure} implies for Lebesgue, almost all initial conditions, $x\in \mathbb{R}^N$, will distribute themselves, according to the physical measure, $\mu$. A typical chaotic system will have infinitely many ergodic measures, but only one physical measure. Among all the invariant measures  the system has, one would expect to see physical measure in the simulation of the dynamical system. Set-oriented numerical methods are available for the computation of a physical measure in dynamical systems \cite{Dellnitz00}.
\begin{definition}[Lyapunov exponents] \label{Lyapunov_exponents}
For $x_{n+1}=f(x_n)$,  let
 \begin{equation}
L(x)=\lim_{n\rightarrow \infty}\left[\left(\prod_{k=0}^n\frac{\partial f}{\partial x}(x_k)\right)^{'}\left(\prod_{k=0}^n\frac{\partial f}{\partial x}(x_k)\right)\right]^{\frac{1}{2n}},\;\;x_0=x.\label{Lya_exp}
\end{equation}
If $\lambda^i_{exp}$ are the eigenvalues of $L(x_0)$, then the Lyapunov exponents, $\Lambda^i_{exp}$, are given by
$\Lambda^i_{exp}=\log \lambda^i_{exp}$ for $i=1,\ldots, N$. The maximum Lyapunov exponent can be obtained as the limit of the following quantity,
\begin{eqnarray}
\Lambda_{exp}(x_0)=\lim_{n\to \infty}\frac{1}{n}\log \parallel \prod_{k=1}^n \frac{\partial f}{\partial x}(x_k)\parallel, \label{max_Ly_exp}
\end{eqnarray}
where $\parallel \cdot \parallel$ is the induced $2-$ norm.
Furthermore, if $\det(L(x))\neq 0$ then
\begin{align}
\label{det_sum_exp}
\lim_{n\to \infty}\frac{1}{n}\log|\det\left(\prod_{k=0}^n \frac{\partial f}{\partial x}(x_k)\right)| =\log \prod_{k=1}^N \lambda^k_{exp}(x).
\end{align}
\end{definition}

\begin{remark} \label{remark_det}
The technical condition for the existence of Lyapunov exponents and the limit in (\ref{det_sum_exp}) is given by the Oseledec multiplicative ergodic theorem \cite{Ruelle85}. Lyapunov exponents for nonlinear systems are defined with respect to a particular invariant measure. Hence, in general, they will be a function of the initial condition, $x_0$. Under the assumption the  system has unique physical invariant measure, the Lyapunov exponents and limits in (\ref{max_Ly_exp}) and (\ref{det_sum_exp}) will be independent of the initial condition, $x_0$.
%For nonlinear systems evolving on a compact state space,
The technical conditions required by Oseledec multiplicative ergodic theorem are satisfied by the system map, $f$, in the form of Assumption \ref{assumption_1}. For a detailed statement of the multiplicative ergodic theorem, we refer  readers to \cite{ergodic_theory_walter} (Theorem 10.4) and discussion in section D of \cite{Ruelle85}.
 %The existence of at least one invariant measure is guaranteed for nonlinear systems evolving on a compact state space \cite{ergodic_theory_walter} (Corollary 6.9.1). Furthermore, every invariant measure also admits ergodic decomposition \cite{ergodic_theory_walter} (Remarks pp. 153), \cite{Mane} (Theorem 6.4).
 \end{remark}
\begin{assumption}\label{unique_physical} Assume  the system, $x_{n+1}=f(x_n)$ has a unique physical invariant measure with all its Lyapunov exponents positive.
\end{assumption}
\begin{remark} The  assumption on uniqueness of physical invariant measure is not necessary to prove the main results of this section. The main results can be proven under the existence of an ergodic invariant measure (or ergodic decomposition of invariant measure) guaranteed to exist following  Remark \ref{remark_det}. However, the assumption of a unique physical measure allows us to prove the main results that are independent of the initial conditions. Without the uniqueness assumption, the main result of this section will be a function of the initial condition, $x$. The assumption of all Lyapunov exponents positive is similar to the assumption made for the case of LTI systems  all eigenvalues are positive. The physical measure with a positive Lyapunov exponent implies the existence of an attractor set with chaotic dynamics. Furthermore, the assumption of all Lyapunov exponent positive is a technical assumption and is made for the simplicity of presentation. This assumption can be relaxed using the technique of tempered transformation \cite{Katok}, which allows one to decompose the state space into directions of positive and negative exponents.
%
%The assumption of unique physical measure allows us to prove the main results of this section that are independent of initial condition. This uniqueness of physical measure assumption is not necessary and the main results can be proved only under the assumption of ergodic invariant measure, which is guaranteed to exists following Remark \ref{remark_det}. However, the results in that case will be function of initial conditions.
\end{remark}
%Assumption \ref{assume_expand} combined with the assumption that the state space $X$ is compact has following consequence on system dynamics.
%\begin{consequence}\label{consequence}
%Consider the dynamical system $x_{n+1}=f(x_n)$, with the system mapping $f:X\to X$ uniformly expanding and the state space $X$ compact. Then the system has unique ergodic invariant measure which is equivalent to Lebesgue measure \cite{Mane}.
%\end{consequence}
%\begin{remark} The proof of the following theorem use the fact that the system has a unique ergodic measure which is equivalent to Lebesgue. Note that the assumption \ref{assume_expand} is sufficient condition for the consequence \ref{consequence} to be true but not necessary. For nonlinear systems evolving on compact space the existence of at least one ergodic invariant measure is guaranteed \cite{Mane}. The uniform expanding property of the system ensures that there is a unique ergodic measure but more importantly this unique ergodic measure is equivalent to Lebesgue. A measure $\mu_1$ is said to be equivalent to $\mu_2$ ($\mu_1\approx \mu_2$) if $\mu_1(A)=0$ if and only if $\mu_2(A)=0$. Ergodictiy with respect to Lebesgue measure has the consequence that any positive Lebesgue measure set can be evolved forward to intersect any other positive measure set with the intersection having positive measure as well.  A trivial consequence of the uniform expanding property of the system map is that all the Lyapunov exponents of the system will be positive.
%\end{remark}
\begin{theorem}\label{theorem_compact_case} Consider  system (\ref{big_system}) with system mapping, $f$, satisfying Assumptions  \ref{assumption_1}, \ref{assumption_2}, and \ref{unique_physical}. Then, a necessary condition for the mean-square exponential incremental stability with controller mapping satisfying Assumption \ref{assumption_controller}
\begin{itemize}
\item for the number of control inputs $1\leq d<N$ is given by
\begin{equation}
\bar \rho ^d \max\left(\prod_{k=1}^N \left(\lambda^{k}_{\exp}\right)^2,\prod_{k=1}^N \left(|\lambda^{k}_{0}|\right)^2 \right) < 1,\label{necessary_condition}
\end{equation}
where $\bar \rho=1-\rho$, $\lambda^k_{exp}=\exp^{\Lambda^k_{exp}}$ and $\Lambda^k_{exp}>0$ is the $k^{th}$ positive  Lyapunov exponent
of the  system, $x_{n+1}=f(x_n)$, and $\lambda_0^k$ are the unstable eigenvalues of the Jacobian, $\frac{\partial f}{\partial x}(0)$, at the origin.
\item the $N$ input case is given by
\begin{equation}
\bar \rho\max \left((\lambda_{exp}^1)^2,|\lambda_{max}|^2\right)<1,
\label{necessary_condition_compact_Ninput}
\end{equation}
where $\lambda^1_{exp}=\exp^{\Lambda^1_{exp}}$ and $\Lambda^1_{exp}>0$ is the maximum  Lyapunov exponent
of the  system, $x_{n+1}=f(x_n)$, and $\lambda_{max}$ is the maximum eigenvalue of the Jacobian $\frac{\partial f}{\partial x}(0)$.

\end{itemize}
\end{theorem}
\begin{proof}
 For $1\leq d<N$: Using the results from Theorem \ref{theorem_support}, a necessary condition for mean-square exponential incremental stability is given by,
\begin{eqnarray}
A'(x_n)Q_0(x_{n+1})A(x_n)- \rho A'(x_n)Q_0(x_{n+1})B(B'Q_0(x_{n+1})B)^{-1}B'Q_0(x_{n+1})A(x_n)<Q_0(x_n).
\label{ll1}\end{eqnarray}
$Q_0(x)$ satisfies the following Riccati-like equation,
\begin{eqnarray}A'(x_n)Q_0(x_{n+1})A(x_n)-A'(x_n)Q_0(x_{n+1})B(I+B'Q_0(x_{n+1})B)^{-1}B'Q_0(x_{n+1})A(x_n)\nonumber\\
+R(x_n)=Q_0(x_n),
\end{eqnarray}
where $x_{n+1}=f(x_n)$ and $A(x_n):=\frac{\partial f}{\partial x}(x_n)$.
Taking determinants on both the sides of (\ref{ll1}) and using the matrix determinant formula, i.e., $\det (I_N- \rho B(B'Q_0(x_{n+1})B)^{-1}B'Q_0(x_{n+1}))=\det (I_d- \rho(B'Q_0(x_{n+1})B)^{-1}B'Q_0(x_{n+1})B)$, we obtain
%Using the matrix determinant formula, i.e., $\det(X+ac^{'})=\det(X)(1+c^{'}X^{-1}a)$, where $a$ and $c$ are column vectors, we obtain
\begin{eqnarray}
\bar \rho^d\det(A^2(x_n))\det(Q_0(x_{n+1}) Q^{-1}_0(x_n))<1.\label{cond}
\end{eqnarray}
The above necessary condition holds true for all initial conditions, $x_n$. Hence, we  can evaluate the above condition along the system trajectory, $x_{n+1}=f(x_n)$. We obtain
\begin{eqnarray}
\bar \rho^{dn} \det(Q_0(x_{n+1}))\prod_{k=0}^n\det(A^2(x_k)) \det(Q^{-1}_0(x_0))<1.
\end{eqnarray}
Taking the logarithm and average with respect to $n$ and in the limit as $n\to \infty$ of the above expression, we obtain
\[\log \bar \rho^d+\lim_{n\to \infty}\frac{1}{n}\log \left(\det(Q_0(x_{n+1}))\prod_{k=0}^n\det(A^2(x_k)) \det(Q^{-1}_0(x_0))\right)<0. \]
Now, using the fact that $Q_0(x)$ is bounded above and below, and using equality (\ref{det_sum_exp}) from Definition \ref{ergodic_measure}, we obtain the following necessary condition for stability,
\begin{eqnarray}
\bar \rho^d\prod_{k=1}^N (\lambda_{exp}^k)^2<1.\label{cond1}
\end{eqnarray}
Similarly, if we evaluate the necessary condition (\ref{cond}) at the equilibrium point, $x=0$, we obtain
\begin{eqnarray}
\bar \rho^d\prod_{k=1}^N (|\lambda_0^k|)^2<1.\label{cond2}\end{eqnarray}
Combining (\ref{cond1}) and (\ref{cond2}), we obtain the required necessary condition (\ref{necessary_condition}) for stability. \\
{\it $N$ input case}:  With the $N$ input case, the matrix, $B$, is invertible. The necessary condition (\ref{necc_condition11}) for mean-square exponential incremental stability reduces to
\begin{eqnarray}
\bar \rho A'(x_n)Q(x_{n+1})A(x_n)<Q(x_n).\label{ss2}
\end{eqnarray}
Since $Q(x_{n})$ is both bounded above and below, there exists an invertible matrix $T(x_n)$, such that  $Q(x_n)=T'(x_n)T(x_n)$. The necessary condition (\ref{ss2}) can be written as
\[\bar \rho \hat A'(x_n)\hat A(x_n)<I,\]
where $\hat A(x_n):=T(x_{n+1})A(x_n)T^{-1}(x_n)$. Since the above inequality holds true for almost all initial conditions, we can evaluate the inequality along the trajectory of the system $x_{n+1}=f(x_n)$. We obtain
\[\bar \rho^n \hat A'(x_0)\hat A'(x_1)\cdots \hat A'(x_{n-1})\hat A(x_{n-1})\cdots \hat A(x_1)\hat A(x_0)<I,\]
which is equivalent to
\[\bar \rho^n (T^{-1}(x_0))^{'}A'(x_0)A(x_1)\cdots A(x_{n-1})Q(x_n)A(x_{n-1})\cdots A(x_1)A(x_0)T^{-1}(x_0)<I.\]
Again, using the fact $Q(x)$ is bounded above and below, the above inequality implies the following necessary condition for stability,
\[\log \bar \rho+ \lim_{n\to \infty}\frac{1}{n}\log \parallel \prod_{k=0}^n A(x_k)\parallel^2<0. \]
Using equality (\ref{max_Ly_exp}) from Definition \ref{ergodic_measure}, we obtain the following necessary condition for stability
\begin{eqnarray}
\bar \rho(\lambda^1_{exp})^2<1, \label{cond_N1}
\end{eqnarray}
where $\lambda_{exp}^1$ is the maximum Lyapunov exponent of the system, $x_{n+1}=f(x_n)$. Similarly, evaluating the necessary conditions for stability at the equilibrium point, $x=0$, we obtain \begin{eqnarray}
\bar \rho|\lambda_{max}|^2<1.\label{cond_N2}
\end{eqnarray}
Combining (\ref{cond_N1}) and (\ref{cond_N2}), we obtain the desired necessary condition (\ref{necessary_condition_compact_Ninput}) for stability of the $N$ input case.

\end{proof}
The results of Theorem \ref{theorem_compact_case} can be viewed as the generalization of the results from Theorem \ref{linear_systems_cond} for LTI systems, where Lyapunov exponents emerge as the natural generalization of linear system eigenvalues to nonlinear systems.

\subsection{Sufficient condition for stability}
Theorem \ref{theorem_sufficient} provides sufficient condition for stability.
\begin{theorem}\label{theorem_sufficient} System (\ref{big_system}), satisfying Assumptions \ref{assumption_1} and \ref{assumption_2}, is mean-square exponential incrementally stable for all values of non-erasure probability, $p$, that satisfies
\begin{eqnarray}
\frac{\partial f(x)}{\partial x}^{'}\left(P-p PB(B'PB)^{-1}  B'P\right)\frac{\partial f(x)}{\partial x}<P,\label{sufficient}
\end{eqnarray}
for almost all with respect to the Lebesgue measure, $x\in X$ and for $x=0$, where matrix $P>0$ satisfies the following Riccati equation,
\[A'PA-A'PB(I+B'PB)B'PA=P,\]
with $A=\frac{\partial f}{\partial x}(0)$.
\end{theorem}
\begin{proof}
Let $k(x)=-(B'PB)^{-1}B'Pf(x)$ be the feedback control input. With this input, we obtain $x_{n+1}=(I-\xi_nB(B'PB)^{-1}B'P)f(x_n)$. Consider any two points, $x_0$ and $y_0$, in $X$ and the curve joining these two points given by $z_0(\alpha)=\alpha x_0+(1-\alpha)y_0$ with $\alpha\in[0,1]$. The dynamical evolution  of this curve,  $z_0(\alpha)$, is given by $z_{n+1}(\alpha)=(I-\xi_n B(B'PB)^{-1}B'P)f(z_n(\alpha))$. Define a new variable, $\zeta_n=\frac{\partial z_n}{\partial \alpha}$. Hence, we obtain \[\zeta_{n+1}=\frac{\partial z_{n+1}}{\partial \alpha}=(I-\xi_n B(B'PB)^{-1}B'P)\frac{\partial f}{\partial z}(z_n)\zeta_n.\]
The length of the curve, $z_0(\alpha)$, after $n$ iterations is given by $\int_0^1 \parallel \zeta_n \parallel d\alpha$. Defining new coordinates, $\rho_n=P^{\frac{1}{2}}\zeta_n=T\zeta_n$, we obtain
\[\rho_{n+1}=T(I-\xi_n B(B'PB)^{-1}B'P)\frac{\partial f}{\partial z}(z_n) T^{-1}\rho_n.\]
Using (\ref{sufficient}), we obtain
\[E_{\xi_0^n}\left[\parallel \rho_{n+1}\parallel^2 \right]\leq \beta^n \parallel \rho_0 \parallel^2\]
for some $\beta<1$. Using the fact there exist positive constants, $\tau_1$ and $\tau_2$, such that $\tau_1 I\leq P\leq \tau_2 I$, we obtain
\[E_{\xi_0^n}\left[\parallel \zeta_{n+1}\parallel^2 \right]\leq \beta^n \frac{\tau_2}{\tau_1} \parallel \zeta_0 \parallel^2=K\beta^n  \parallel \zeta_0 \parallel^2.\]
We have
\[E_{\xi_0^n}\left[\parallel x_{n+1}-y_{n+1} \parallel^2\right]\leq \int_0^1 E_{\xi_0}^n[ \parallel \zeta_{n+1}\parallel^2 ]d\alpha\leq K \beta^n \parallel x_0-y_0\parallel^2.\]

\end{proof}

\section{Example}\label{section_simulation}

In this section, we continue with the discussion on the synchronization in discrete Lorentz system started in Section \ref{section_motivation}.  As mentioned previously, there are two competing instabilities the controller needs to work against to stabilize the unstable equilibrium at the origin. The parameter values, $\alpha$ and $\beta$, can be chosen so  the instability at the equilibrium point $(0,0)$ is larger or smaller than the average instability on the attractor set captured by the Lyapunov exponents. In particular, for a parameter value of $\alpha=1.25$ and $\beta=0.75$, the eigenvalues for the linearization at the origin are $\lambda_1=1.28$ and $\lambda_2=1.28$ with the product equal to $1.65$. The positive Lyapunov exponent for these parameter values is equal to $\Lambda_{exp}=.34$. Hence, $\lambda_{exp}=\exp{0.34}=1.40<\lambda_1 \lambda_2=1.65$. So, the critical erasure probability is computed, based on the unstable eigenvalues, and  equals $p_{eig}^{*}=1-\frac{1}{\lambda_1^2\lambda_2^2}=0.63$.

For the parameter values, $\alpha=2.25$ and $\beta=0.29$, the product of the unstable eigenvalues at the origin equals $\lambda_1 \lambda_2=1.08$ and the exponential of the positive Lyapunov exponent is equal to $\lambda_{exp}=1.13$. This leads to the critical non-erasure probability, based on the positive Lyapunov exponent  equal to $p^{*}_{lya}=1-\frac{1}{\lambda_{exp}^2}=0.21$. The critical non-erasure probability, based on the unstable eigenvalues, is $p_{eig}=0.14$.
The controller $k(x)$ is designed to cancel the nonlinearity (i.e., $u_n=k(x_n)=-[(1+\alpha \beta)(x_n+\sqrt{\alpha})-\beta (x_n+\sqrt{\alpha}) (y_n+\alpha) -\sqrt{\alpha}]$. In Fig. \ref{lorentz1}a, we show the attractor set for the uncontrolled system for the parameter value of $\alpha=2.25$ and $\beta=0.29$. The attractor set for the parameter value of $\alpha=1.25$ and $\beta=0.75$ is already shown in Fig. \ref{sim_section2}a. The attractor set in Figs. \ref{sim_section2}a and \ref{lorentz1}a correspond to the physical measure of the Lorentz system for two different sets of parameter values.

\subsection{Simulations} In Fig. \ref{lorentz1}, we show the simulation results for the parameter value of $\alpha=2.25$ and $\beta=0.29$, corresponding to the case where the instability due to the Lyapunov exponent is dominant over the eigenvalues.
In Figs. \ref{lorentz1}b and \ref{lorentz1}c, we show the histogram for the state error plot between the two trajectories of the system (\ref{slave}) for non-erasure probability of $p=0.15(<0.21=p^*)$ and for $p=0.35$ respectively. This simulation results in Fig. \ref{lorentz1} are obtained for the case where $F_s\neq F_m$. Similarly, in Figs. \ref{lorentz2}a and \ref{lorentz2}b, we show the histogram of the error plots for the case of $F_s=F_m$.
%individual state error between the master and slave system state dynamics averaged over $500$ different realizations of random variable, $\xi_n$, for non-erasure probability, $p=0.15$, below the critical non-erasure probability of $p^*=0.21$. In Fig. \ref{lorentz1}b, we show the state error plot for the non-erasure probability of $p=0.28$.

Comparing Fig.\ref{lorentz1}a with Fig. \ref{lorentz1}b and Fig. \ref{lorentz2}a with \ref{lorentz2}b, we notice the error dynamics show larger fluctuation around  zero for non-erasure probability below the critical value of $p^*$. It is important to emphasize that the non-erasure probability, $p=0.15$, satisfies $p_{eig}=0.14<0.15<0.21=p^*_{lya}$. Hence, although the system has overcome local instability due to eigenvalues, the global instability of the attractor set is still a limiting factor for stabilization. This is further evident from comparing the two attractor sets obtained for the $p=0.15$ and $p=0.28$ in Fig. \ref{lorentz2}c. These two attractor sets are obtained by simulating the system Eq. (\ref{slave}) without the coupling from  system (\ref{master}).
%Simulation result for the attractor set in Fig. \ref{lorentz2}a is obtained by adding a uniform random variable, $\varpi_n\in[0,0.01]$ to state $x_{n+1}$.
From this plot, we see that the attractor set for $p=0.15$ (in red) is chaotic, but for $p=0.28$, the attractor set (in blue) is concentrated around the origin. Finally, in Figs. \ref{lorentz3}a and \ref{lorentz3}b, we show the plot for the linearized error covariance for the parameter values of $(\alpha=1.25,\beta=0.75)$, and $(\alpha=2.25,\beta=0.29)$ respectively. From these plots, we clearly see that the covariance for parameter values of $(\alpha=1.25,\beta=0.75)$ and $(\alpha=2.25,\beta=0.29)$ become unbounded for non-erasure probability of $p^{*}=0.21$ and $p^{*}=0.63$, respectively, as predicted by the main result of this paper.
The predicted value of critical non-erasure probability of $p^*=0.63$ is in agreement with the simulation results presented in section \ref{section_motivation}.

\begin{figure}[h] \centering \subfigure[]
{\includegraphics[width=2.1in]{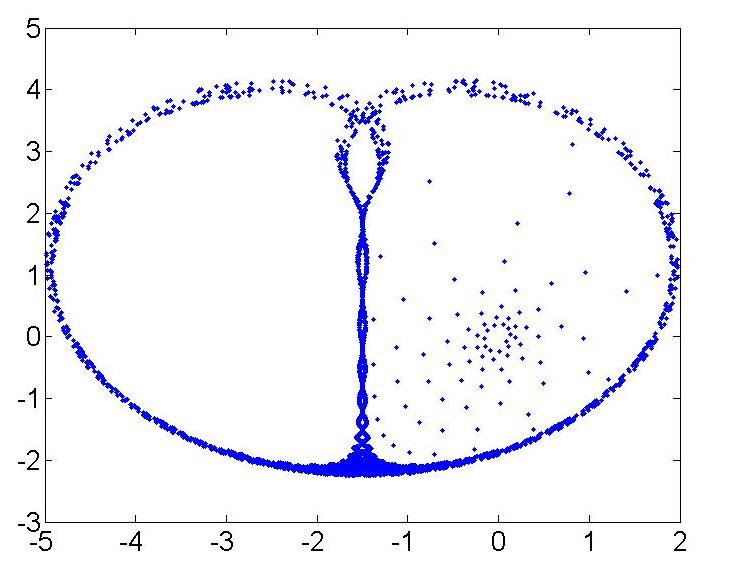}}\subfigure[]
{\includegraphics[width=2.1in]{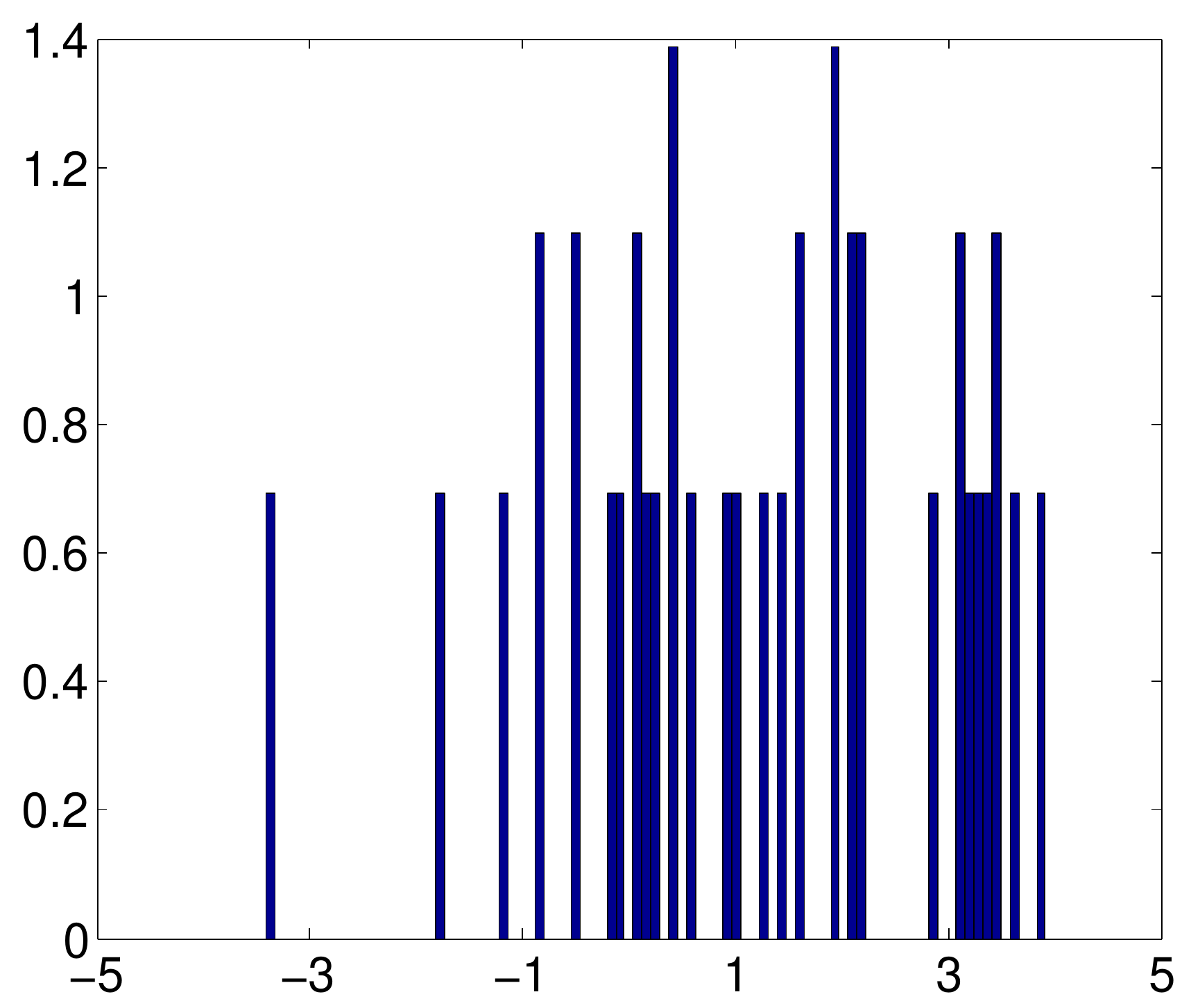}}
\subfigure[]
{\includegraphics[width=2.1in]{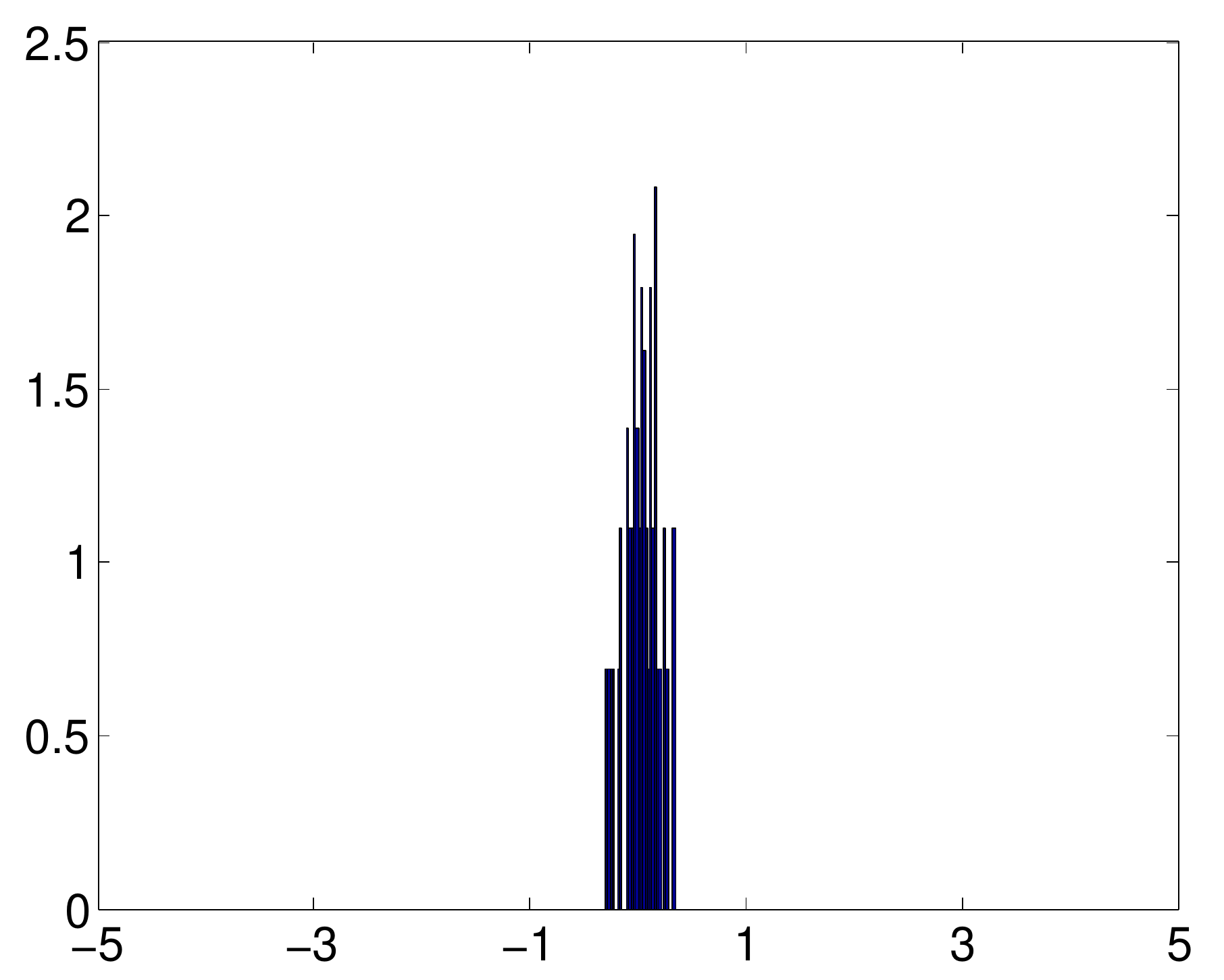}}
\caption{a) Chaotic attractor for parameters $\alpha=2.25, \beta=0.29$); b) Histogram for the error dynamics between two trajectories for non-erasure probability of $p=0.15<p^{*}=0.21$; c)
Histogram for the error dynamics between two trajectories for non-erasure probability of $p=0.3$.}
\label{lorentz1}
\end{figure}
\begin{figure}[h] \centering \subfigure[]
{\includegraphics[width=2.2in]{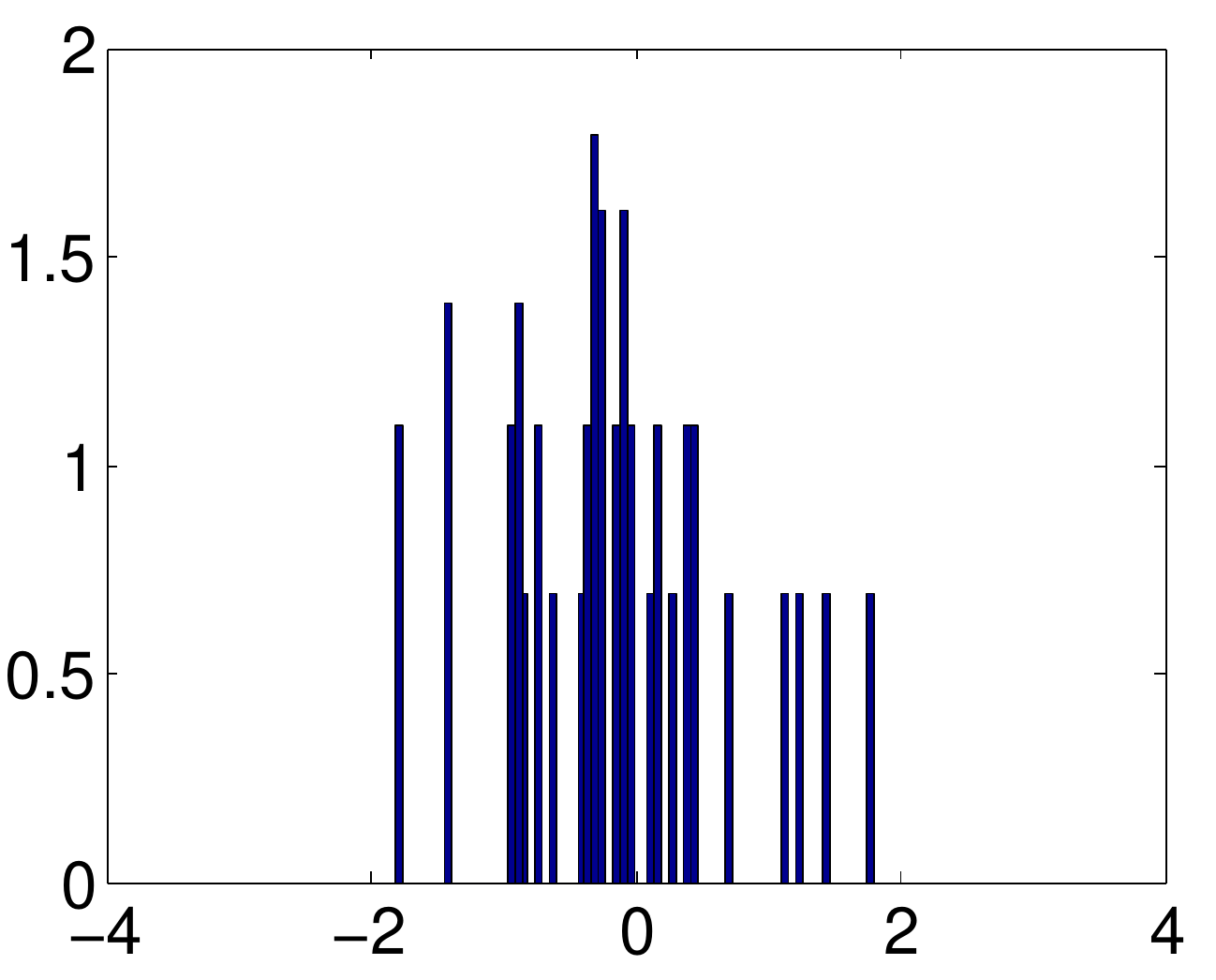}}\subfigure[]
{\includegraphics[width=2.1in]{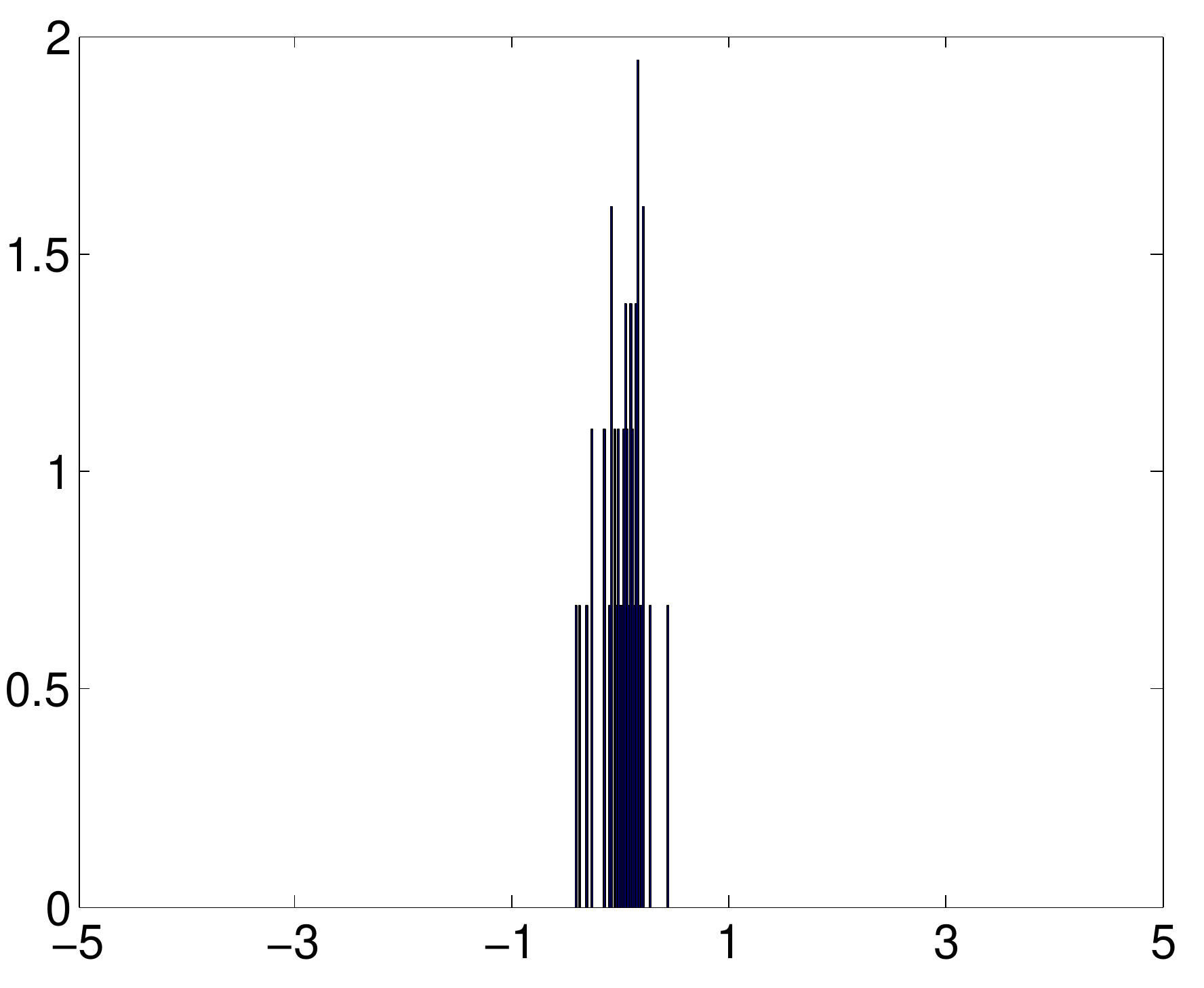}}
\subfigure[]
{\includegraphics[width=2.1in]{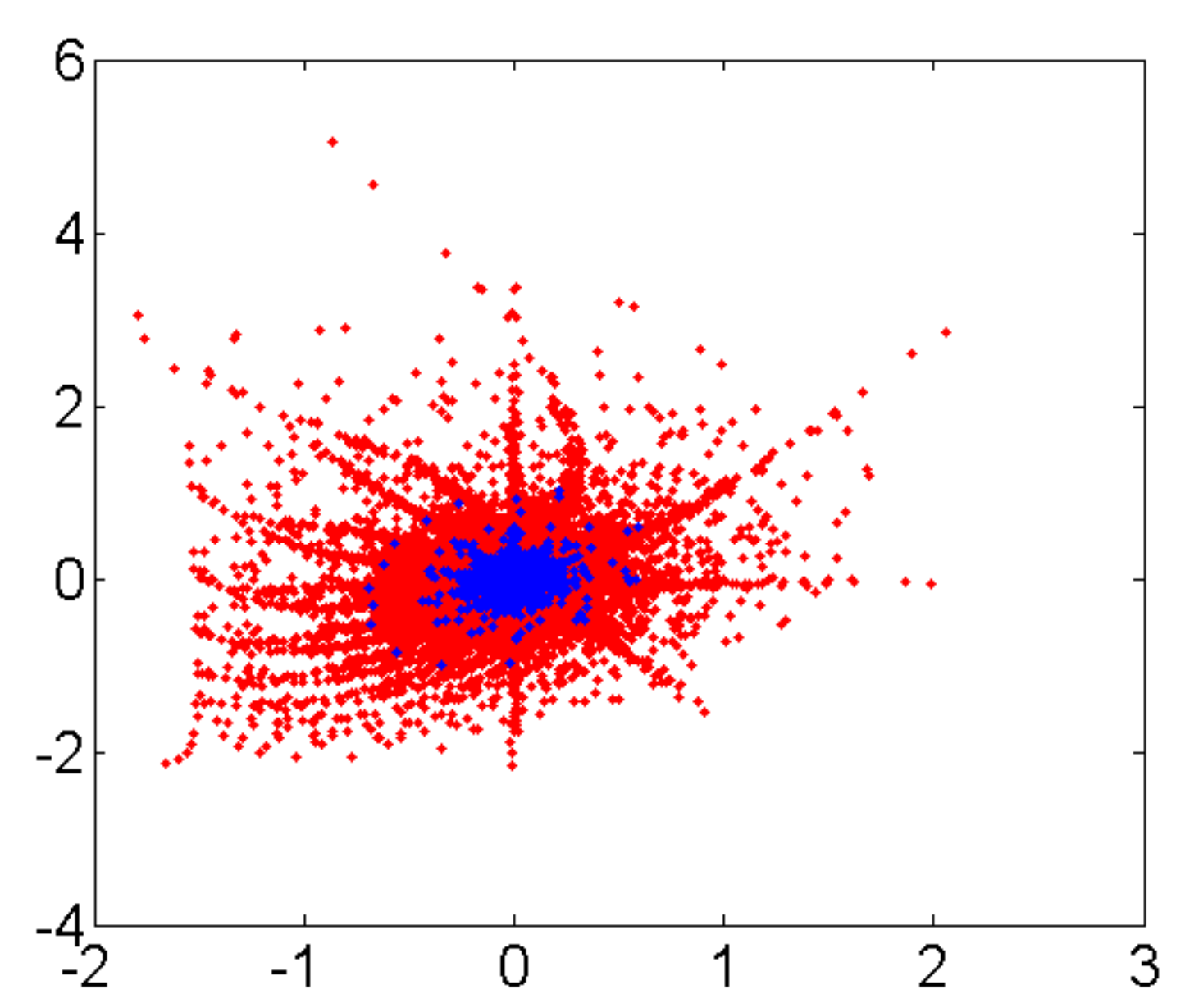}}
\caption{a)  Histogram for the error dynamics between two trajectories for non-erasure probability of $p=0.15$; b) Histogram for the error dynamics between two trajectories for non-erasure probability of $p=0.3$; c)  Comparison of two attractor sets for $p=0.15$ (red) and $p=0.28$ (blue); }
\label{lorentz2}
\end{figure}
\begin{figure}[h] \centering
%\subfigure[]
%{\includegraphics[width=2.2in]{attractor_comparison_p015_p030_a001case2.pdf}}
\subfigure[]
{\includegraphics[width=2.1in]{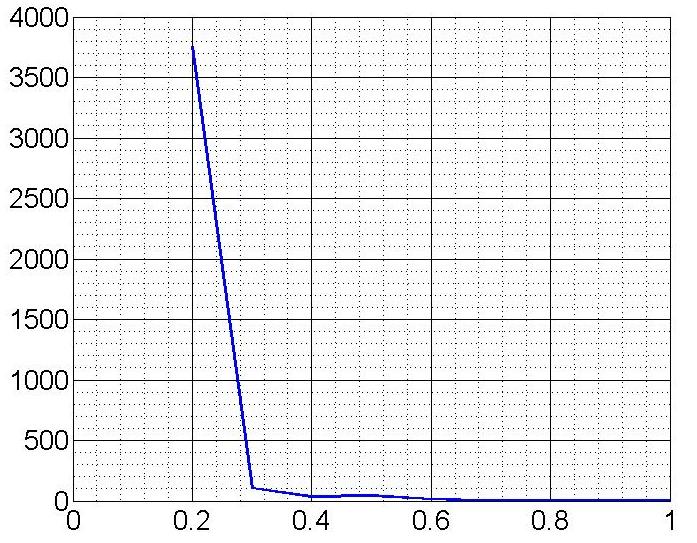}}
\subfigure[]
{\includegraphics[width=2.1in]{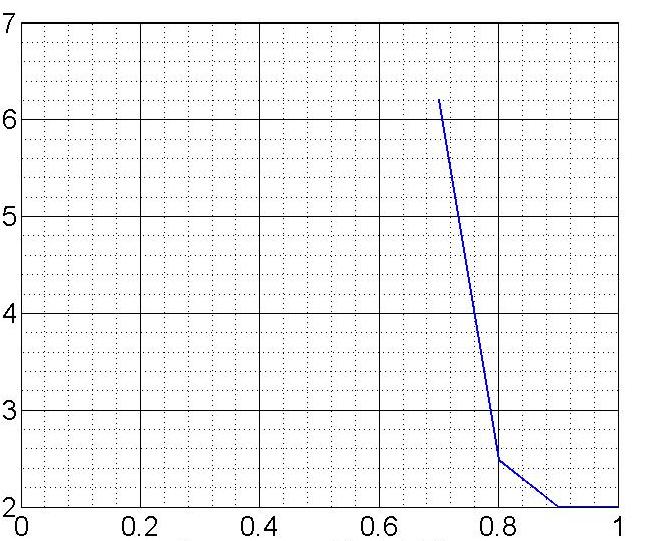}}
\caption{a) Trace of the linearized covariance vs non-erasure probability $p$ for $(\alpha,\beta)=(2.25,0.29)$;  b) Trace of the linearized covariance vs non-erasure probability $p$ for $(\alpha,\beta)=(1.25,0.75)$. }
\label{lorentz3}
\end{figure}

\section{Conclusions}\label{section_conclusion}
In this paper, we  considered the problem of incremental stabilization of a class of nonlinear systems by state feedback controller,  when the actuation command may be lost on a communication link with certain probability. The stochastic notion of mean-square exponential stability is adopted to study the incremental stabilization problem over an uncertain communication link. One of the important features of our main results is its global nature away from equilibrium and the emergence of the open-loop Lyapunov exponents as the natural generalization of the  linear system eigenvalues in capturing the system limitations. Our results are quite encouraging and our novel approach, based on ergodic theory of dynamical system, is amenable to various extensions.

\bibliographystyle{IEEEtran}
\bibliography{ref,ref1}

\section{Appendix}\label{appendix}

Proof of Theorem \ref{theorem_secondmoment}
\begin{proof}

From the definition of mean square exponential incremental stability of system (\ref{big_system}) and using the fact that $f(0)=k(0)=0$ (Assumption \ref{assumption_1} and \ref{assumption_controller}) it follows that  system 
\begin{eqnarray}
x_{n+1}=f(x_n)+\xi_n Bk(x_n)=:T(x_n,\xi_n)\label{system_feedback_nonoise}.
\end{eqnarray}
%is also mean square exponential incremental stable. 
%
% we know that system $x_{n+1}=F(x_n,\xi_n,w_n)$ is mean square exponentially incrementally
%stable. Since $w_n$ is exogenous control input, then it follows that system $x_{n+1}=F(x_n,\xi_n,0)$ obtained by
%making $w_n\equiv 0$ is also mean square exponentially incrementally stable. 
%From the mean square exponential incremental stability of (\ref{system_feedback_nonoise}) and using the fact that
%$T(0,\xi_n)\equiv 0$, it follows that  (\ref{system_feedback_nonoise}) 
is mean square exponentially stable i.e., there exist positive constants $\bar K_1<\infty$ and $\bar \beta_1<1$ such that 
\begin{eqnarray}E_{\xi_0^n}[\parallel x_{n+1}\parallel^2]\leq \bar K_1 \bar \beta_1^n \parallel
x_0\parallel^2.\label{2moment}
\end{eqnarray}
%{\textcolor{red}{where the constants $K_1=\frac{K(2-\beta)}{(1+\beta)}>1$ and $\beta_1=1-\frac{1}{K_1}<1$ were
%obtained from Part 1 on Theorem 2.7.}}
Furthermore, following Holder's Inequality we also obtain
\begin{eqnarray}E_{\xi_0^n}[\parallel x_{n+1}\parallel] \leq \left (E_{\xi_0^n}[\parallel x_{n+1}\parallel^2]\right)^{\frac{1}{2}} \leq \left(\bar K_1 \bar \beta_1^n \parallel
x_0\parallel^2 \right)^{\frac{1}{2}}= \bar K_1^{\frac{1}{2}} \bar \beta_1^{\frac{n}{2}} \parallel
x_0\parallel.\label{1moment}\end{eqnarray}
The mean square exponential stability of (\ref{system_feedback_nonoise}) implies existence of Lyapunov function,
$V(x)$, satisfying following conditions
\[c_1\parallel x_n\parallel^2 \leq V(x_n)\leq c_2 \parallel x_n\parallel^2,\;\;
E_{\xi_n}[V(T(x_n,\xi_n))]\leq  \alpha V(x_n)\]

\[ \parallel\frac{\partial V}{\partial x}(x_n)\parallel \leq c_3 \parallel x_n\parallel, \]

%\[c_1 x_n\leq \nabla V(x_n) \leq c_2 x_n,\;\;\; k_1\leq \parallel \nabla^2 V\parallel \leq k_2, \]
where, $c_1=1, c_2,c_3$ are positive constants and
$\alpha<1$. This can be proved as follows. Consider the following construction of Lyapunov function.
\[V(x_n)= \parallel x_n\parallel^2+\sum_{k=n+1}^\infty E_{\xi_n^{k-1}}[\parallel x_k\parallel^2]\implies
V(x_{n+1})= \parallel x_{n+1}\parallel^2+\sum_{k=n+2}^\infty E_{\xi_{n+1}^{k-1}}[\parallel x_k\parallel^2]\]
where $x_{n+1}=T(x_n,\xi_n)$. From the construction of $V$, it is clear that $V(x_n)\geq  \parallel x_n
\parallel^2$, hence $c_1=1$.
The upper bound follows from the mean square exponentially stability of (\ref{system_feedback_nonoise}) as
follows
\[V(x_n)= \parallel x_n\parallel^2+\sum_{k=n+1}^\infty E_{\xi_n^{k-1}}[\parallel x_k\parallel^2]\leq \parallel
x_n\parallel^2+\sum_{k=n+1}^\infty \bar K_1 \bar \beta_1^{k-1-n} \parallel x_n\parallel^2\leq \parallel x_n\parallel^2 c_2
, \]
where $c_2:=\bar K_1(1+\frac{1}{1-\bar \beta_1})>1$. From the construction of $V$, it follows that
\[E_{\xi_n}[V(x_{n+1})]=V(x_n)-\parallel x_n\parallel^2\leq (1-\frac{1}{c_2})V(x_n)=\alpha V(x_n)\]
since $c_2>1$, $(1-\frac{1}{c_2})=:\alpha<1$. We next prove the bound on the derivative of
Lyapunov function $V$. Towards this goal we make use of the Assumptions \ref{assumption_1} and \ref{assumption_controller} providing uniform bounds for
the Jacobian of system mapping, $f$, and feedback controller, $k$. We have

\[\left |\frac{\partial T}{\partial x}(x,\xi)\right|=\left|\frac{\partial f}{\partial x}(x)+\xi B \frac{\partial
k}{\partial x}(x)\right|\leq L(1+|\xi|)\]
where $L$ is the assumed uniform bound on the Jacobian of $f$ and $Bk$ i.e., $\max\{\parallel \frac{\partial f}{\partial x}\parallel, \parallel B\frac{\partial k}{\partial x}\}\leq L$. We introduce following notation
\[x_2=T(T(x_0,\xi_0),\xi_1)=:T^2(x_0,\xi_0^1), \;\;{\rm hence},\;\; x_k=T^k (x_0,\xi_0^{k-1})\] We now have
\[\frac{\partial V}{\partial x}(x)=2 x +2\sum_{k=1}^\infty  E_{\xi_0^{k-1}}\left[\left[\frac{\partial T}{\partial
x}\left(T^{k-1}(x,\xi_0^{k-2}),\xi_{k-1}\right)\right]^\top T^k(x,\xi_0^{k-1})\right].\]
Hence,
\begin{eqnarray}
\parallel\frac{\partial V}{\partial x}(x)\parallel \leq 2\parallel x\parallel +2 L\sum_{k=1}^\infty
E_{\xi_0^{k-1}}\left[(1+|\xi_{k-1}|)\parallel x_k\parallel\right].\label{ee1}
\end{eqnarray}
Now we use  following using Holder's inequality
\begin{eqnarray}
E_{\xi_0^{k-1}}\left[|\xi_{k-1}|\parallel x_{k}\parallel\right]\leq \left(E_{\xi_0^{k-1}}\left[| \xi_{k-1}|
^2\right]\right)^{\frac{1}{2}} \left(E_{\xi_0^{k-1}}\left[\parallel x_{k} \parallel
^2\right]\right)^{\frac{1}{2}}=(\sigma^2+\mu^2)^{\frac{1}{2}}\left(E_{\xi_0^{k-1}}\left[\parallel x_{k} \parallel
^2\right]\right)^{\frac{1}{2}}\label{ee2}
\end{eqnarray}
Combining inequalities (\ref{2moment}), (\ref{1moment}) and (\ref{ee2}), we obtain
\[
\parallel\frac{\partial V}{\partial x}(x)\parallel \leq 2\parallel x\parallel +2 L\sum_{k=1}^\infty
E_{\xi_0^{k-1}}\left[\parallel x_k\parallel\right]+2L(\sigma^2+\mu^2)^{\frac{1}{2}}\sum_{k=1}^\infty
\left(E_{\xi_0^{k-1}}\left[\parallel x_k\parallel^2\right]\right)^{\frac{1}{2}}\]
\[\parallel\frac{\partial V}{\partial x}(x)\parallel \leq  \left(2 +2 L\sum_{k=1}^\infty
\bar K_1^{\frac{1}{2}}\bar \beta_1^{\frac{k-1}{2}} +2L(\sigma^2+\mu^2)^{\frac{1}{2}}\sum_{k=1}^\infty
\bar K_1^{\frac{1}{2}}\bar \beta_1^{\frac{k-1}{2}}\right)\parallel x\parallel =: c_3
\parallel x\parallel\]

%We now consider system (\ref{system_feedback_noise2}) in the rest of the proof. Using the Lyapunov function,
%following inequality holds true for system (\ref{system_feedback_noise2})
 We now consider system
 \[x_{n+1}=T(x_n,\xi_n)+ \gamma_n\] in the rest of the proof. Using the Lyapunov function, following inequality
 holds true for the above system

 \[E_{\xi_0^n\gamma_0^n}[\parallel x_{n+1}\parallel^2]\leq \frac{1}{c_1}
 E_{\xi_0^n\gamma_0^n}[V(x_{n+1})]=\frac{1}{c_1}E_{\xi_0^n\gamma_0^n}[V(T(x_n,\xi_n)+ \gamma_n)].\]
We now apply Mean Value Theorem to expand, $V$, as follows
\[V(T(x_n,\xi_n)+\gamma_n)=V(T(x_n,\xi_n))+\frac{\partial V}{\partial x}(T(x_n,\xi_n)+c_n \gamma_n)\gamma_n\]
 for some $c_n\in [0,1]$, which in general function of $\gamma_n$.
Using the fact that $c_n$ is bounded by one and using Holder's inequality we obtain
\begin{eqnarray}E_{\xi_n\gamma_n}[V(T(x_n,\xi_n)+\gamma_n)]&\leq
E_{\xi_n}[V(T(x_n,\xi_n))]+c_3E_{\xi_n\gamma_n}[\parallel T(x_n,\xi_n)\parallel \parallel
\gamma_n\parallel]+c_3 E_{\gamma_n}[\parallel \gamma_n\parallel^2]\nonumber\\&\leq
E_{\xi_n}[V(T(x_n,\xi_n))]+c_3C^{\frac{1}{2}}\left(E_{\xi_n}[\parallel
T(x_n,\xi_n)\parallel^2]\right)^{\frac{1}{2}}+c_3C
\end{eqnarray}
Now using the  fact that $\parallel x\parallel^2\leq V(x_n), E_{\xi_n}[V(T(x_n,\xi_n))]\leq \alpha V(x_n)$, we obtain
\[E_{\xi_n\gamma_n}[V(T(x_n,\xi_n)+\gamma_n)]\leq \alpha V(x_n)+b_1\left(
V(x_n)\right)^{\frac{1}{2}}+\bar Q\]
where $b_1:=c_3C^{\frac{1}{2}}\alpha^{\frac{1}{2}}$ and $\bar Q:=c_3 C$. Now choose $\alpha_0<1-\alpha$. Using inequality of arithmetic and geometric means (AM-GM) we obtain
\[\left(\frac{b_1^2\alpha_0}{\alpha_0} V(x_n)\right)^{\frac{1}{2}}\leq \frac{\alpha_0}{2}V(x_n)+\frac{b_1^2}{2\alpha_0}\]
Using the AM-GM inequality, we write
\[E_{\xi_n\gamma_n}[V(T(x_n,\xi_n)+\gamma_n)]\leq (\alpha +\frac{\alpha_0}{2}) V(x_n)+\frac{b_1^2}{2\alpha_0}+ \bar Q\]
Since $\alpha_0<1-\alpha$, we have $\beta_2:=\alpha+\frac{\alpha_0}{2}<1$ and defining $Q:=\bar Q+\frac{b_1^2}{2\alpha_0}<\infty $, we obtain
\[E_{\xi_n\gamma_n}[V(T(x_n,\xi_n)+\gamma_n)]\leq \beta_2 V(x_n)+ Q\]
Taking expectation over the sequence $\{\xi_{n-1},\ldots,\xi_0\}$ and $\{\gamma_{n-1},\ldots, \gamma_0\}$ and using induction, we obtain
\begin{eqnarray}
E_{\xi_0^n,\gamma_0^n}[\parallel x_{n+1} \parallel^2]&\leq& E_{\xi_0^n\gamma_0^n}[V(T(x_n,\xi_n)+\gamma_n)]\leq \beta_2^{n+1} V(x_0)+Q\sum_{k=0}^n\beta_2^k\nonumber\\
&\leq & \beta_2^n c_2 \parallel x_0\parallel^2+\frac{Q}{1-\beta_2}
\end{eqnarray}

%\end{enumerate}
\end{proof}

Proof of Theorem \ref{linearization}.
\begin{proof}
%Consider the following system
%\begin{eqnarray}
%x_{n+1}=f(x_n)+\xi_n Bk(x_n)+\xi_n w_n\label{system_appendix}
%\end{eqnarray}
From the mean square exponential incremental stability of system (\ref{big_system}) we know that following is true 
\begin{eqnarray}
x_{n+1}=f(x_n)+\xi_n Bk(x_n)+\xi_n B w_n\label{system_appendix}
\end{eqnarray}%
%Since system (\ref{system_feedback_noise}) is assumed to be mean square exponentially incremental system, system (\ref{system_appendix}) obtained as a special case of system (\ref{system_feedback_noise}) with $\gamma_n\equiv 0$ is also mean square exponentially incrementally stable i.e.,
\[E_{\xi_0^n}[\parallel x_{n+1}-y_{n+1}\parallel^2]\leq K_1 \beta_1^n \parallel x_0-y_0\parallel^2\]
for $x_0,y_0\in \mathbb{R}^N$ and for some positive constants $K_1$ and $\beta_1<1$. Since, $w_n$ is exogenous input,   let $w_n=- k(z_n)$ in (\ref{system_appendix}) with $z_n$ evolving according to the dynamics $z_{n+1}=f(z_n)$. Hence, we have

\begin{eqnarray}
z_{n+1}&=&f(z_n)\label{observer_system1}\\
x_{n+1}&=&f(x_n)+\xi_n Bk(x_n)-\xi_n B k(z_n).\label{observer_system}
\end{eqnarray}
Since system (\ref{observer_system}) is mean square exponentially incrementally stable, hence all the trajectories of
system (\ref{observer_system}) will converge to each other.
One particular trajectory of the system (\ref{observer_system}), say $\bar x_n$, is obtained by taking the initial condition $\bar x_0=z_0$, which gives us $\bar x_n=z_n$ for all $n\geq 0$. Hence, from the mean-square exponential incremental stability of system (\ref{observer_system}), it follows
\begin{eqnarray}
E_{\xi_0^n}\left[ \parallel x_{n+1}-z_{n+1}\parallel^2\right]=E_{\xi_0^n}\left[ \parallel x_{n+1}-\bar x_{n+1}\parallel^2\right]\leq K_1 \beta_1^n \parallel x_0 -\bar x_0\parallel^2= K \beta^n \parallel  x_0 -z_0\parallel^2,\label{eqnstab}
\end{eqnarray}
for $\forall n\geq 0$ and Lebesgue almost all initial conditions, $x_0, z_0 \in \mathbb{R}^N$ and, in particular, for $z_0=0$. Using  (\ref{observer_system1}) and (\ref{observer_system}), we obtain

\begin{eqnarray}
e_{n+1}:=z_{n+1}-x_{n+1}=(f(z_n)+\xi_n B k(z_n))-(f(x_n)+\xi_n B k(x_n)).\label{error}
\end{eqnarray}
 Define ${\cal A}(z_n,\xi_n):=\frac{\partial f}{\partial z}(z_n)+\xi_n B \frac{\partial k}{\partial z}(z_n)$. Then, using (\ref{error}) and the Mean value theorem for a vector valued function , we obtain
\begin{eqnarray}
e_{n+1}&=&\int_0^1 {\cal A}(z_n-se_n,\xi_n)ds e_n=\left(\int_0^1 {\cal A}(z_n-se_n,\xi_n)ds\right) \cdots \left(\int_0^1 {\cal A}(z_0-se_0,\xi_0)ds \right)e_0 \nonumber\\
&=:&\prod_{k=0}^n \left(\int_0^1 {\cal A}(z_k-se_k,\xi_k)ds \right)e_0.
\end{eqnarray}
Note, $e_n$ is a function of $e_0, x_0$, and the random sequence $\xi_0^{n-1}$. Hence, we define ${\cal B}_n(z_0,e_0,\xi_0^{n-1}):=\int_0^1{\cal A}(z_n-se_n,\xi_n)ds$ and ${\cal B}_0^n(z_0,e_0,\xi_0^{n-1}):=\prod_{k=0}^n \int_0^1 {\cal A}(z_k-s e_k,\xi_k)ds$. Using the above definitions and mean-square exponential incremental stability property and (\ref{eqnstab}), we know  there exist positive constants, $K_1<\infty$ and $\beta_1<1$, such that
\[E_{\xi_0^n}\left [\parallel e_{n+1} \parallel^2\right]=e_0' E_{\xi_0^n}\left[{\cal B}_0^n(z_0,e_0,\xi_0^{n})' {\cal B}_0^n(z_0,e_0,\xi_0^{n})\right]e_0 \leq K_1 \beta_1^n e_0'e_0.\]
Note that the above inequality holds true for any positive scaling of $e_0$. Let $\alpha_\ell$ be the scaling of $e_0$, where $\{\alpha_\ell\}$ be the sequence such that $\lim_{\ell\to \infty}\alpha_\ell=0$. We then have
\begin{eqnarray}
e_0' E_{\xi_0^n}\left[{\cal B}_0^n(z_0,\alpha_\ell e_0,\xi_0^{n})' {\cal B}_0^n(z_0,\alpha_\ell e_0,\xi_0^{n})\right]e_0 \leq K_1 \beta_1^n e_0'e_0.\label{eqn50}
\end{eqnarray}

Let ${\cal B}_0^n (z_0,e_o,\xi_0^{n})_{ij}$ denote the $i^{th}$ row and $j^{th}$ column entry of the matrix ${\cal B}_0^n(z_0,e_0,\xi_0^{n})$.
% since the above inequality is true for all values of $e_0$.  Hence,  the initial error is scaled by $\alpha_\ell$, where $\{\alpha_\ell\}$ is such that $
%\lim_{\ell \to \infty} \alpha_\ell=0$.
 From Assumptions \ref{assumption_1} and \ref{assumption_controller}, we know that both $\frac{\partial f}{\partial z}$ and $\frac{\partial k}{\partial z}$ are uniformly bounded. Hence, for any fixed $n$, we apply Dominated Convergence Theorem to the function ${\cal A}_{ij}(z_n-\alpha_\ell e_n,\xi_0^n)$. ${\cal A}_{ij}$ denotes the $ij$ entry of matrix ${\cal A}$. We obtain
\[\lim_{\ell \to \infty} {\cal B}_n(z_0,\alpha_\ell e_0,\xi_0^n)_{ij}=\lim_{\ell \to \infty}\int_0^1 {\cal A}_{ij}(z_n-s\alpha_\ell e_n,\xi_0^n)ds= {\cal A}_{ij}(z_n,\xi_0^n)= {\cal B}_n(z_0,0,\xi_0^n)_{ij}.\]
The above argument holds true for all entries of matrix ${\cal A}$. Therefore, we obtain
\begin{eqnarray}\lim_{\ell\to \infty} {\cal B}_0^n(z_0,\alpha_\ell e_0, \xi_0^{n})={\cal B}_0^n(z_0,0, \xi_0^{n}).\label{eqnn}
\end{eqnarray}
For every fixed $n$, consider the sequence of functions $e_0'{\cal B}_0^n (z_0,\alpha_\ell e_0,\xi_0^n)'{\cal B}_0^n (z_0,\alpha_\ell e_0,\xi_0^n)e_0$, where ${\cal B}_0^n (z_0,e_0,\xi_0^n)=\prod_{k=0}^n \left(\int_0^1 {\cal A}(z_k- se_k,\xi_0^n)ds \right)$. We apply Fatou's Lemma to exchange limit with the expectation to obtain,
\begin{eqnarray}
&e_0'E_{\xi_0^n}\left[\lim_{\ell\to \infty} {\cal B}_0^n(z_0,\alpha_\ell e_0,\xi_0^n)'{\cal B}_0^n(z_0,\alpha_\ell e_0,\xi_0^n)\right]e_0\nonumber\\&\leq \lim_{\ell \to \infty}e_0'E_{\xi_0^n}\left[{\cal B}_0^n(z_0,\alpha_\ell e_0,\xi_0^n)'{\cal B}_0^n(z_0,\alpha_\ell e_0,\xi_0^n)\right]e_0 \leq K_1 \beta_1^n e_0'e_0.\label{eqnn2}
\end{eqnarray}

%
%By the dominated convergence theorem \cite{Rudin} and continuity of ${\cal A}(x_n-se_n,\xi_n)$, we have $\lim_{\ell \to \infty} {\cal B}_n(x_0,\alpha_\ell e_0,\xi_0^n)_{ij}={\cal B}_n(x_0,0,\xi_0^n)_{ij}$, which implies $\lim_{\ell \to \infty} {\cal B}_n(x_0,\alpha_\ell e_0,\xi_0^n)={\cal B}_n(x_0,0,\xi_0^n)$. Hence, we have \vspace{-0.15in}
%\begin{eqnarray}\lim_{\ell\to \infty} {\cal B}_0^n(x_0,\alpha_\ell e_0, \xi_0^{n})={\cal B}_0^n(x_0,0, \xi_0^{n}).\label{eqnn}
%\end{eqnarray}
%
%\begin{eqnarray}
%&e_0'E_{\xi_0^n}\left[\lim_{\ell\to \infty} {\cal B}_0^n(x_0,\alpha_\ell e_0,\xi_0^n)'{\cal B}_0^n(x_0,\alpha_\ell e_0,\xi_0^n)\right]e_0\leq \lim_{\ell \to \infty}e_0'E_{\xi_0^n}\left[{\cal B}_0^n(x_0,\alpha_\ell e_0,\xi_0^n)'{\cal B}_0^n(x_0,\alpha_\ell e_0,\xi_0^n)\right]e_0 \nonumber\\ &\leq K \beta^n e_0'e_0.\label{eqnn2}
%\end{eqnarray}
Using (\ref{eqn50}), (\ref{eqnn}), and (\ref{eqnn2}), we obtain
\[e_0'E_{\xi_0^n}\left[{\cal B}_0^n(z_0,0,\xi_0^n)'{\cal B}_0^n(z_0,0,\xi_0^n)\right]e_0\leq K \beta^n e_0'e_0,\]
where ${\cal B}_0^n(z_0,0,\xi_0^n)$ is the product of Jacobian matrices, ${\cal A}(z_n,\xi_n)=\frac{\partial f}{\partial z}(z_n)+\xi_n B\frac{\partial k}{\partial z}(z_n)$, computed along the trajectory of the system $z_{n+1}=f(z_n)$. Hence, we obtain
\[E_{\xi_0^n}\left[e_0' \left(\prod_{k=0}^n {\cal A}(z_k,\xi_k)\right)'\left(\prod_{k=0}^n {\cal A}(z_k,\xi_k)\right)e_0\right]\leq K_1 \beta_1^n e_0'e_0,\;\;\;\forall n\geq 0.\]
Since the matrices in the above equation are independent of $e_0$, we can write

\begin{eqnarray}
E_{\xi_0^n}\left[\eta_0' \left(\prod_{k=0}^n {\cal A}(z_k,\xi_k)\right)'\left(\prod_{k=0}^n {\cal A}(z_k,\xi_k)\right)\eta_0\right]\leq K_1 \beta_1^n \eta_0'\eta_0,\;\;\;\;\forall n\geq 0,\label{equation_linear}
\end{eqnarray}
where the evolution of $\eta_n$ is governed by $\eta_{n+1}=\left(\frac{\partial f}{\partial z}(z_n)+\xi_n\frac{\partial k}{\partial z}(z_n)\right)\eta_n$ and $z_n$ by $z_{n+1}=f(z_n)$.   Since (\ref{equation_linear}) is true for almost all initial condition $z_0\in X$ and in particular for $z_0=0$, we prove the statement of the Theorem after relabeling the state $z_n$ to $x_n$.
\end{proof}

\end{document}